\documentclass[11pt]{amsart}

\usepackage[colorlinks,citecolor=blue,urlcolor=blue,bookmarks=true]{hyperref}
\usepackage{marginnote}

\usepackage[utf8]{inputenc}
\usepackage[english]{babel}
\usepackage{csquotes}
\usepackage{amsmath}
\usepackage{amsfonts}
\usepackage{amssymb}
\usepackage{setspace}
\usepackage{amsthm}
\usepackage{graphicx}
\usepackage{float}
\usepackage{blindtext}
\theoremstyle{plain}
\newtheorem{thm}{Theorem}[] 
\newtheorem{cor}[]{Corollary}
\newtheorem{lem}[]{Lemma}

\theoremstyle{definition}
\newtheorem{definition}[]{Definition} 
\newtheorem{assumption}[]{Assumption}
\newtheorem{remark}[]{Remark}
\newcommand{\convdistr}{\stackrel{\mathcal{D}}{\rightarrow}}

\title{
Exact asymptotics
of ruin probabilities with linear
Hawkes arrivals}
\date{\today}

\author{Zbigniew Palmowski}
\address{
		Zbigniew Palmowski\\
        Wrocław University of Science and Technology\\
        Wyb. Wyspia\'{n}skiego 27\\
        50-370 Wroc\l{}aw\\
        Poland
}
\email{zbigniew.palmowski@pwr.edu.pl}

\author{Simon Pojer}
\address{Simon Pojer\\
      Institute of Statistics\\
      University of Technology Graz\\
      Kopernikusgasse 24/III, 8010 Graz\\
      Austria
}
\email{ simon.pojer@tugraz.at}

\author{Stefan Thonhauser}
\address{Stefan Thonhauser\\
  Institute of Statistics\\
      University of Technology Graz\\
      Kopernikusgasse 24/III, 8010 Graz\\
      Austria
}
\email{stefan.thonhauser@tugraz.at}

\begin{document}
\selectlanguage{english}

\begin{abstract}
In this paper we determine bounds and exact asymptotics of the ruin probability
for risk process with arrivals given by a linear marked Hawkes process.
We consider the light-tailed and heavy-tailed case of the claim sizes. Main technique
is based on the principle of one big jump, exponential change of measure, and renewal arguments.

\vspace{3mm}

\end{abstract}

\maketitle

\pagestyle{myheadings} \markboth{\sc Z.\ Palmowski  --- S.\ Pojer --- S. Thonhauser}
{\sc
Ruin probabilities with Hawkes arrivals}

\section{Introduction}

In this paper we consider the probability of ruin for a risk process
\begin{equation}\label{CLprocess} X_t := u+ct-\sum_{i=1}^{N_t}U_i,\end{equation}
where $u$ is a positive initial capital and $c$ the premium rate.
In addition, claim sizes $\{ U_i \}_{i\in \mathbb N}$ are assumed to be i.i.d. with cumulative distribution function $F_U$,
positive, and independent of the arrival process $\{N_t\}_{t\geq 0}$. We assume that arrivals are determined by a Hawkes process $\{N_t\}_{t\geq 0}$.
We recall that it means that $\{N_t\}_{t\geq 0}$ is a simple point process with intensity process
\[ \lambda_t := a + \sum_{T_i\leq t} h(t-T_i, Y_i),\]
where $a$ is constant baseline intensity, $\{Y_i\}_{i\in \mathbb{N}}$ is a sequence of positive i.i.d. random variables called often shocks,
$\{T_i\}_{i\in \mathbb{N}}$ are arrival times of the points in $\{ N_t \}_{t \geq 0}$
and
$h: \mathbb{R}_{\geq 0} \times \mathbb{R}_{\geq 0} \to \mathbb{R}_{\geq 0}$ is a decay function satisfying
\begin{equation}\label{assh}
 \mathbb{E}\left[ \int_0^\infty h(t,Y) \mathrm{d}t\right]=:\mu <1\quad \text{and}\quad \mathbb{E}\left[ \int_0^\infty t\,h(t,Y) \mathrm{d}t\right]<+\infty.\end{equation}
We assume that all considered stochastic quantities are defined
on some filtered probability space $(\Omega, \mathcal{F}, \{\mathcal{F}_t\}_{t\geq 0} , \mathbb{P})$
satisfying the usual conditions. To stress the dependence of the distribution of certain stochastic objects on the initial capital $u$, we write $\mathbb P_u$ for the measure $\mathbb P$ with condition $X_0=u$, and $\mathbb E_u \left[\cdot\right] $ for the corresponding expectation.

Introducing the Hawkes arrival process into risk theory is important  due to its
a self-exciting structure in which recent events affect the intensity of claim
occurrences more than distant ones.
In other words, every jump of the Hawkes process increases
its intensity; hence, the probability of further jumps.
In fact a Hawkes process has another crucial interpretation, namely,
the cluster of claims come to the insurer according to a homogeneous Poisson process $\{\tilde{N}_t\}_{t\geq 0}$
with intensity $a$ and aforementioned finite (thanks to assumption \eqref{assh}) cluster is produced according to a Galton-Watson branching process
with offspring distribution whose mean is $\mu<1$.
In other words, one 'claim' event that arrives according to the Poisson process $\{\tilde{N}_t\}_{t\geq 0}$
produces a chain of smaller claims that are reported to insurance company in time.

The ruin probability lies in the center of risk theory since the classical Cram\'er-Lundberg model, introduced in 1903 by Filip Lundberg \cite{lund} and then generalised in the 1930's  by Harald Cram\'er (see \cite{grenander95}).
It equals the probability that the reserves of the company will ever go below zero, that is,
\[\psi(u):= \mathbb{P}_u(\tau<+\infty)\]
for the ruin time
\begin{equation}\label{ruintime}
\tau:= \inf\{t\geq 0: X_t<0\}.
\end{equation}
Usually (and in this paper) the so-called net profit condition is assumed, that is,
\begin{equation}\label{net}
\lim_{t\to+\infty}X_t=+\infty \qquad {\rm a.s.},\end{equation}
hence $\psi(u)<1$.
Ruin is a technical term - it does not mean that the company goes bankrupt at the ruin time. In fact, downward crossing of some fixed level (that could be treated as level zero)
by reserves of the insurance company forces this company to do some actions to make the business profitable. Estimation of the ruin probability is also required for solvency purposes and it can be used,
after setting this ruin probability to an acceptably low level, to calculate the premium that should be charged.

As we said, in most of the cases only a bound or an estimate (for large values of initial reserves) is sufficient in daily practice. This is a main goal of this paper.

At the beginning we compare our process to the modified one by shifting in time all cluster claims to 'cluster' events.
This produces two  statements. After light-trailed assumptions specified later we prove in Corollary \ref{corbounds} that
there exist positive constants $C_-$, $C_+$ and an adjustment coefficient $R>0$ such that
\[C_-e^{-Ru} \leq \psi(u) \leq C_+ e^{-Ru}.\]
Similar result was presented in \cite{Albrecher.2006}
for the  Cox claim arrival process with a Poisson shot
noise intensity. Their proof has the following flaw though.
In this paper it is assumed incorrectly that the batch claim producing a ruin event is always the “last one”. In other words, every batch claim (i.e. cluster) consists of several single claims which appear delayed in time. \cite{Albrecher.2006} assumed that if we consider all claims of all clusters which appear until the ruin time of the Cram\'er-Lundberg process,
the last singular claims (last in terms of real time) belongs always to the "last" (hence the same) cluster.
This is not correct since this would mean that it can not happen that e.g. the second cluster has a claim which occurs after the third cluster
has appeared. We bypassed this problem by considering the maximum of the length of all clusters which appeared until ruin of the Cram\'er-Lundberg
process.


Under complementary assumption that claim sizes are strongly subexponential (hence heavy-tailed) we prove in Corollary \ref{aymptoticssubexponential} that
\[\lim_{u \to \infty} \frac{\psi(u)}{1-F^s_{U}(u)}=C_h\]
for a constant $C_h$ identified explicitly, where
\begin{equation}\label{residual}
F^s_{ U}(u):=\frac{1}{\mathbb{E}\left[U_1\right]} \int_0^u (1-F_U(y)) \mathrm{d}y
\end{equation}
is the integrated tail distribution.
We used here the principle of one big jump as well.
The above results appeared in Proposition 13 of \cite{MR3130449}.
Our proof is different though not
requiring tedious checking of Assumptions 1 of
\cite{MR3130449}.

To be able to find the exact asymptotics of the ruin probability $\psi(u)$ in the light-tailed case, we additionally assume that
Hawkes $N_t$ is Markovian, that is,
$$\lambda_t := a + (\lambda -a)e^{-\beta t} + \sum_{i=1}^{N_t} Y_i e^{-\beta(t-T_i)} $$
for some $\lambda>a>0$ and $\beta>0$.
Using exponential change of measure and renewal arguments, we show that in this case
\begin{equation}\label{expasymp} \psi(u) e^{Ru} \to C,\qquad \text{as $u\rightarrow+\infty$}\end{equation}
for a positive constant $C$.
In the proof, we show few other facts (like, e.g. the form of the generator of $\{ (X_t, \lambda_t, t)\}_{t \geq 0}$)
that are of own interest. In particular we find sufficient conditions
for the intensity process to be positive Harris recurrent and we prove that the corresponding recurrence times are light-tailed.
We would like to underline that \eqref{expasymp} is stronger than Theorem 4.1 of \cite{Stabile} who derived
only the logarithmic asymptotic showing that $\lim_{u\to+\infty}\frac{1}{u}\ln \psi(u)=-R$.

To prove all aforementioned main results we have to built up
new theory for the risk process with the Hawkes arrival process
and for linear Hawkes process itself.
In particular, the exponential change of measure
that we apply seems to be new.
We also find the parameters of the risk process after this change of measure. To do so, in the Markov set-up,
we find the extended generator of the process $\{(X_t,\lambda_t, t)\}_{t\geq 0}$  and its domain. Again, to our best knowledge, this
is a new result. We also introduce the counterpart of the classical Lundberg equation which has attracted lots of attention in ruin theory.

In our construction, it is crucial to show that the intensity process $\{\lambda_t\}_{t\geq 0}$ is positive recurrent under original and new measures. We use these recurrence epochs of the intensity process combined with ladder epochs of the risk process to construct new time points $\{\phi_i\}_{i\in \mathbb{N}}$ that allow to
formulate
a renewal-type equation for the ruin probability multiplied be an appropriate exponential function appearing on the right-hand side of \eqref{expasymp}.

To prove that the limit of the solution of this renewal-type
equation exists and is finite, we find sufficient conditions for the directly Riemann integrability of some functions appearing in this equation.  This is possible due to detailed analysis of exponential moments of $X_{\phi_1}$.

We believe that ideas of proofs presented in this paper can be used for other cluster arrival processes as well.

The paper is organised as follows.
In Section \ref{sec:general} we construct a modified risk process and show exponential two-sided bounds for the ruin probability $\psi(u)$.
In this section we also prove the heavy-tailed asymptotiocs of $\psi(u)$.
Section \ref{sec:Markov} is devoted to study Markovian Hawkes process and
proving the Cram\'er-Lundberg asymptotics for $\psi(u)$ using renewal arguments and the change of measure technique.
We finish our paper by detailed analysis of the case where the shocks and claims are exponentially distributed
(see Section \ref{sec:example}).

\section{Cluster representation and modified risk process}\label{sec:general}
\subsection{Cluster representation of the risk process}
In this paper we consider a risk process driven by a linear marked Hawkes process $\{ N_t\}_{t \geq 0}$
satisfying \eqref{assh}.
For convenience, we will omit the properties 'linear' and 'marked' for the rest of the paper and refer to this process as Hawkes process.
In this section, the so-called \textit{Poisson cluster representation} or simply cluster representation is crucial.
For this, we consider the influence of a single shock event due to the baseline intensity $a$, called a base event. If such a jump occurs, the counting process $\{ N_t\}_{t \geq 0}$ increases by $1$, and the intensity increases by $h(0,y)$, where $y$ is a realisation of a random variable with distribution function $F_Y$. This increase of the intensity triggers ${\rm Poi}(\int_0^\infty h(t,y) \mathrm{d}t)$ further jumps, which we denote by \textit{children}, where ${\rm Poi}(\alpha)$ denotes the random variable with Poisson distribution with a parameter $a$. The occurrence of a single child increases the counting process by $1$ and triggers again ${\rm Poi}(\int_0^\infty h(t,y') \mathrm{d}t)$ new jumps, where $y'$ is a new, independent realisation with the same distribution. We collect the offspring of a base event and call it \textit{cluster}. An important question which appears naturally is if such a  cluster consists of finitely many points or may explode. As \cite{Basrak.2019} showed, the number of points $\kappa$ in such a generic cluster coincides with the number of points in a subcritical Galton-Watson branching process. By this, we have that $\kappa$ is finite almost surely and has finite expectation $\mathbb{E}\left[\kappa\right]= \frac{1}{1-\mu}$.

By this, we can rewrite the intensity process as \[\lambda_t = a + \sum_{i=1}^{\bar N_t} \sum_{j=0}^{K_{t-\bar T_i}^{(i)}} h\left(t-(\bar T_i + \sum_{k=1}^j T_{ik}), Y_{ij}\right),\]
where $\{ \bar N_t \}_{t \geq 0}$ denotes a Poisson process with constant intensity $a$ and jump times $\bar T_i$. For fixed $i$, the process $\{ K^{(i)}_t\}_{t \geq 0}$
counts the offspring of the $i$-th base event and its jump times are given by $\bar T_i+T_{i1}, \bar T_i+ T_{i1}+T_{i2}, \ldots$, where $T_{ij}$ denotes the inter-jump time between the $j-1$-th and $j$-th jump of cluster $i$. The random variables $Y_{ij}$ correspond to the shock of the $j$-th event in the $i$-th cluster, and $Y_{i0}$ is the shock due to the base event of the $i$-th cluster.
Since every cluster has the same distribution, we have that for all $j\geq 1$, the sequence $\{ \sum_{k=1}^j T_{ik} \}_{i \geq 1}$ is i.i.d.
By the same procedure, we can rewrite the surplus process as \[ X_t = u+ct - \sum_{i=1}^{ \bar N_t} \sum_{j=0}^{K^{(i)}_{t-\bar T_i}} U_{ij}, \] where the random variable $U_{ij}$ denotes the claim due to the $j$-th event of the $i$-th cluster.

Since the number of events of a single cluster is finite almost surely, we have that $K^{(i)}_t \to K^{(i)}_\infty=: \kappa_i$, almost surely as $t \to \infty$. Since all clusters have the same distribution, we have that  $\{ \kappa_i\}_{i \in \mathbb{N}}$ is an i.i.d. sequence of random variables satisfying
\begin{equation}\label{meankappa}\mathbb{E}\left[\kappa_i\right]=\frac{1}{1-\mu};\end{equation}
see e.g. p. 203 of \cite{Daryl}.

This representation shows us the main feature of our model. An underlying Poisson process triggers with every jump a cluster consisting of a random number $\kappa_i$ of claims. These claims do not occur immediately but are delayed in time. We will use this representation to derive pathwise bounds for the surplus process.

\subsection{Upper and lower bounds for the surplus}
To derive a lower bound for the surplus process we ignore the mentioned delay in time. We define the clustered process $\{ \tilde X_t\}_{t \geq 0}$ as
\begin{equation}\label{CLshifted}
\tilde X_t = u+ct - \sum_{i=1}^{\bar N_t} \sum_{j=0}^{\kappa _i} U_{ij} =: u+ct - \sum_{i=1}^{\bar N_t} \tilde U_i, \end{equation}
 where
 \begin{equation}\label{tildeU}
 \tilde U_i := \sum_{j=0}^{\kappa _i} U_{ij},\qquad i \in \mathbb{N}
 \end{equation}
 form an i.i.d. sequence independent of the counting Poisson process $\{ \bar N_t\}_{t \geq 0}$, counting the number of clusters. Since $\kappa_i \geq K^{(i)}_t$ for all $t$ almost surely, we have that
 \begin{equation}\label{comparison}X_t \geq \tilde X_t.\end{equation}
 for any realisation of the arrival and claim processes. 
 To avoid trivial cases, we assume that the net profit condition
 \begin{equation}\label{netprofit} c> \mathbb E [\tilde U] a ,\end{equation} holds for the modified risk process $\{\tilde{X}_t\}_{t\geq 0}$.
 Observe that the clustered process $\{\tilde{X}_t\}_{t\geq 0}$ is now a Cram\'er-Lundberg process. Consequently, we can use standard results for this process to obtain an upper bound for the ruin probability $\psi(u)$ of our surplus process $\{X_t\}_{t\geq 0}$.
 Let
 \begin{equation}\label{ruincluster}\tilde \psi(u):=\mathbb{P}_u(\inf_{t\geq 0} \tilde{X}_t<0)\end{equation}
 be a ruin probability of the corresponding clustered process $\{\tilde{X}_t\}_{t\geq 0}$.
 \begin{lem}\label{upperboundgen}
 We have, \[ \psi(u) \leq \tilde \psi(u).\]
 \end{lem}
\begin{proof}
 This follows immediately by inequality \eqref{comparison}.
 \end{proof}

To derive a lower bound for the ruin probability, we follow the ideas of \cite{Albrecher.2006}, who consider a general shot-noise model, and find a suitable constant $l_1$ such that $\psi(u-l_1) \geq \tilde \psi(u)$. For this, we first introduce some additional notation. We write $L_i$ for the
length of cluster $i$, i.e.
\begin{equation*}\label{defL}L_i = \inf \left \lbrace \, t \geq 0 \, \left \vert \, K^{(i)}_t = K^{(i)}_\infty = \kappa_i \, \right. \right \rbrace.\end{equation*}
By the i.i.d. structure of the clusters, we have that the sequence $\{ L_i \}_{i \in \mathbb{N}}$ is also i.i.d. following some cumulative distribution function $F_L$. Further, we observe that $L_i = \sum_{j=1}^{\kappa_i} T_{ij}$, where $\{ T_{ij} \}_{j \in \{ 1, \ldots \kappa_i \} }$ denotes the set of inter-jump times of the $i$-th cluster.

Let
\[\tilde \tau:=\inf\{t\geq 0: \tilde{X}_t<0\}\]
denote the time of ruin of the clustered process $\{ \tilde X_t \}_{t \geq 0}$ and $\tilde L$ the time from $\tilde \tau$ until all clusters which appeared up to time $\tilde \tau$ finished, i.e.
\[\tilde L := \sup \left \lbrace \bar T_i + \sum_{j=1}^{\kappa_i} T_{ij}:  \bar T_i \leq \tilde \tau \, \right\rbrace - \tilde \tau.\] In other words, this is the minimal random time such that for all $i$ with $\bar T_i \leq \tilde \tau$ we have $K^{(i)}_{\tilde \tau + \tilde L} = K^{(i)}_{\infty}=\kappa_i$.

Then, we have whenever $\tilde \tau$ and $\tilde L$ are finite that
\begin{align}\label{ineq:tildeL}
&X_{\tilde \tau+ \tilde L} = u+c (\tilde \tau + \tilde L) - \sum_{i=1}^{\bar N_{\tilde \tau + \tilde L}} \sum_{j=0}^{K^{(i)}_{\tilde \tau + \tilde L}} U_{ij} \\&\quad
= u + c \tilde \tau - \sum_{i=1}^{\bar N_{\tilde \tau}} \sum_{j=0}^{\kappa_i} U_{ij} + c \tilde L - \sum_{i= \bar N_{\tilde \tau}+1}^{\bar N_{\tilde \tau + \tilde L}} \sum_{j=0}^{K^{(i)}_{\tilde \tau + \tilde L}} U_{ij} \\&\quad \leq \tilde X_{\tilde \tau} + c \tilde L.
\end{align}
Consequently, \[ X_{\tilde \tau+ \tilde L}-c\tilde L  \leq \tilde X_{\tilde \tau} < 0.\]
Hence, if we reduce the initial capital of our process by $c\tilde L$, then ruin of the clustered process also causes ruin of the original surplus. The main problem is that $\tilde L$ is random, could be infinite and is not measurable with respect to the filtration of the original surplus process $\{X_t\}_{t\geq 0}$ nor with respect to the filtration of the clustered process $\{\tilde X_t\}_{t\geq 0}$. To bypass these problems, we want to get constants $l_1$ and $C$ such that $\psi(u-cl_1) \geq C \tilde \psi(u)$.
\begin{lem}\label{lowerboundgen}
Let $ t>0 $ such that $\mathbb{P}\left[ L_1 \leq t\right]>0$. Then there exists a constant $K_t \in (0, 1]$ such that $\psi(u-ct) \geq K_t \tilde \psi(u)$.
\end{lem}
\begin{proof}
At first, we set for convenience $\tilde L (\omega) = \infty$ whenever $\tilde \tau (\omega) = \infty$.
Then, we observe that for all $\omega$ such that $\tilde L (\omega) <t$ we have that $\tilde \tau_u (\omega)< +\infty$ and $\tau_{u-ct}(\omega) < +\infty$. Here, $\tau_u$ denotes the time of ruin with initial capital $u$ of the surplus process $\{ X_t\}_{t \geq 0}$ and $\tilde \tau_u$ the corresponding ruin time of $\{ \tilde X_t \}_{t \geq 0}$ with starting point $u$.
The implication that $\tilde \tau_u$ has to be finite is clear since $\tilde L=\infty$ if $\tilde \tau = \infty$. On the other hand, we have by inequality (\ref{ineq:tildeL}) that in this case \[\tau_{u-ct}(\omega)\leq \tilde \tau_u(\omega) + \tilde L(\omega) \leq \tilde \tau_u(\omega) + t < +\infty.\] Using these implications, we get for fixed $t$ that \begin{multline*}
    \psi(u-ct) \geq \mathbb{P}\left[ \tau_{u-ct} < +\infty, \tilde L \leq t\right] = \mathbb{P}\left[ \tilde L \leq t\right] \\= \sum_{n=1}^\infty \mathbb{P}\left[ \tilde L \leq t \, \left\vert \, \tilde \tau_u = \bar T_n\right.\right] \mathbb{P}\left[ \tilde \tau_u = \bar T_n\right].
\end{multline*}
If we can now bound $\mathbb{P}\left[ \tilde L \leq t \, \left\vert \, \tilde \tau_u = \bar T_n\right.\right]$ from below by a positive constant $K_t$, this inequality would imply that \[ \psi(u-ct) \geq K_t \sum_{n=1}^\infty \mathbb{P}\left[ \tilde \tau_u =\bar T_n\right] = K_t \tilde \psi(u),\] since ruin for the clustered process can only happen at jump times $\bar T_n$ of the Poisson process $\{\bar N_t\}_{t\geq 0}$.

To do this, we determine the distribution of a sequence of auxiliary variables $\tilde L_n:= \max \left\lbrace \bar T_i-\bar T_n + L_i: i \leq n \right\rbrace$, i.e. the random variable $\tilde L$ conditioned on $\tilde \tau = \bar T_n.$ Then, we have that
\[\mathbb{P}\left[ \tilde L_n \leq t \right] = \mathbb{P}\left[ \bar T_1 - \bar T_n + L_1 \leq t, \ldots, \bar T_{n-1} - \bar T_n + L_{n-1} \leq t, L_n \leq t\right].\] The random variable $L_n$ is independent of the random variables $L_i$ for $i<n$ and of the Poisson process $\bar N_t$. Further, by the lack of memory of exponential law, we have that $E_n:= \bar T_n - \bar T_{n-1}$ is exponentially distributed with parameter $a$ and independent of the information up to $\bar T_{n-1}$. This yields
\[ \mathbb{P}\left[ \tilde L_n \leq t \right] = F_L(t) \mathbb{P}\left[\bar T_1-\bar T_{n-1} \leq t+E_n, \ldots, L_{n-1} \leq t+ E_n \right]. \] Conditional on $E_n$, we have the same structure as before and now $L_{n-1}$ is independent of all other random variables and $E_{n-1} = \bar T_{n-1} -\bar T_{n-2}$ is again exponentially distributed. Consequently, \begin{multline*}
    \mathbb{P}\left[ \tilde L_n \leq t \right] = F_L(t)\times\mathbb{E}\Big[F_L(t+E_n) \\ \mathbb{P}\left[\bar T_1-\bar T_{n-2} \leq t+E_n+E_{n-1}, \ldots, L_{n-2} \leq t+ E_n+E_{n-1} \left\vert \, E_n \right. \right]\Big].
\end{multline*}
Using the fact that $F_L(x) \leq 1$ for all $x>0$, and continuing above procedure
we get that \[ \mathbb{P}\left[ \tilde L_n \leq t \right] = \mathbb{E}\left[ \prod_{i=0}^{n-1} F_L\left(t + \tilde T_i\right) \right] \geq  \mathbb{E}\left[ \prod_{i=0}^{\infty} F_L\left(t + \tilde T_i \right) \right],\]
where $$\tilde T_{k} :=  \sum_{j=1}^{k} E_{n+1-j}.$$
This property holds for all $n\in \mathbb{N}$, and even though the jump times $\tilde T_i$ depend on $n$, the expectations coincide since they have the same distribution and the same dependence structure for all $n$ due to the stationarity of the Poisson process.
Now, we still have to show that this expectation is strictly positive.

To do so, we use that an infinite product $\prod_{n=1}^\infty(1+ a_n)$  converges absolutely to a nonzero real number if the series $\sum_{n=1}^\infty |a_n|$ converges and $1+a_n >0$ for all $n$. If we write $\bar F_L(x)$ for the tail $1-F_L(x)$ we get for fixed $\omega \in \Omega$ that \[\prod_{i=0}^{\infty} F_L\left(t + \tilde T_i(\omega) \right)  = \prod_{i=0}^{\infty} \left(1- \bar F_L\left(t + \tilde T_i(\omega) \right)\right).\] Since for all $i$, we have that $\bar F_L(t+\tilde T_i) \leq \bar F_L(t) < 1$, we want to show that the series $\sum_{i=0}^\infty \bar F_L\left(t + \tilde T_i(\omega) \right)$ converges for almost all $\omega$, to get that the product $\prod_{i=0}^{\infty} F_L\left(t + \tilde T_i \right)$ converges almost surely to a random variable $C \in (0,1]$, which would give us positiveness of the corresponding expectation. \\
By the strong law of large numbers, we have that $\frac{1}{n}\tilde T_n \to a$ almost surely. Let now $\omega$ be such that this convergence holds. Then, there exists a $N_\omega$ such that $\bar T_n  \geq n\frac{a}{2}$ for all $n \geq N_\omega$. By this, and the monotone decreasing behaviour of $\bar F_L$, we get that \begin{align*} \sum_{i=0}^\infty \bar F_L\left(t + \tilde T_i(\omega) \right) &\leq N_\omega +1 + \sum_{i= \leq N_\omega +1}^\infty \bar F_L\left(t + \tilde T_i(\omega) \right)
\\ &\leq N_\omega +1 + \sum_{i= N_\omega +1}^\infty \bar F_L\left(t + n \frac{a}{2} \right).
\end{align*}
The remaining series converges if and only if the corresponding integral $$\int_{N_\omega+1}^\infty \bar F_L\left(t+ x \frac{a}{2}\right) \,\mathrm{d}x,$$ converges. For this, we have that \begin{multline*}
\int_{N_\omega+1}^\infty \bar F_L\left(t+ x \frac{a}{2}\right)\, \mathrm{d}x \leq \int_{N_\omega+1}^\infty \bar F_L\left( x \frac{a}{2}\right)\, \mathrm{d}x \leq \frac{2}{a} \int_0^\infty \bar F_L(y)\, \mathrm{d}y = \frac{2}{a} \mathbb{E}\left[L_1\right].
\end{multline*}
Here, $L_1$ denotes the length of the first cluster. By the proof of Lemma 1 in \cite{Moller.2005}, we have that \[\mathbb{E}[L_1] \leq  \frac{1}{1-\mu}\mathbb{E}\left[ \int_0^\infty t\, h(t,Y)\, \mathrm{d}t\right] < +\infty.\]
Consequently, we have that the series $\sum_{i=0}^\infty \bar F_L\left(t + \tilde T_i \right)$ converges a.s., which implies that the product $\prod_{i=0}^{\infty} F_L\left(t + \tilde T_i\right) \in (0,1]$ almost surely. This gives us finally that $K_t:= \mathbb{E}\left[\prod_{i=0}^{\infty} F_L\left(t + \tilde T_i \right)\right] \in (0,1]$.

\end{proof}

From Lemmas \ref{upperboundgen} and \ref{lowerboundgen}
we can conclude the general case with the following theorem.
\begin{thm}\label{thm:general_case}
    Let $l_1$ be such that $\mathbb{P}\left[ L_1 \leq t\right]>0$. Then, the ruin probability in the general Hawkes model satisfies \[ \tilde \psi(u) \geq \psi(u) \geq \tilde \psi(u+l_1 c),\] where $\tilde \psi(u)$ denotes the ruin probability \eqref{ruincluster} of the clustered process $\{\tilde{X}_t\}_{t\geq 0}$, i.e. the ruin probability of a Cram\'er-Lundberg process.
\end{thm}

\subsection{Cram\'er-Lundberg bounds and heavy-tailed asymptotics}
Interested in the asymptotic behaviour of the ruin probability as the initial capital tends to infinity,
we will see that this behaviour depends highly on the behaviour of the distribution of the generic claim size $U$ of the risk process \eqref{CLprocess}. To  understand this behaviour we recall that the generic claim size of 'shifted'
Cram\'er-Lundberg process \eqref{CLshifted} is given in \eqref{tildeU}, that is,
$\tilde{U}=\sum_{k=1}^\kappa U_k$ for a generic cluster size $\kappa$.
We prove the following basic fact.
\begin{lem}\label{kappalighttailed}
The generic cluster size $\kappa$ is light-tailed, that is, that there exists $\theta >0$
such that $\mathbb{E}\left[e^{\theta \kappa}\right]<+\infty$.
\end{lem}
\begin{proof}
Observe that $\kappa$ has the same law as a total progeny in a Galton–Watson branching process
with Poisson offspring distribution whose mean is $\mu$.
From \cite{MR253433} it follows that
\[\mathbb{P}(\kappa=n)=\frac{1}{n}\mathbb{P}(S_n=-1),\]
where $S_n$ is a random walk with i.i.d. increments $X_i\stackrel{\mathcal{D}}{=} {\rm Poi}(\mu)-1$ for a Poissonian random
variable ${\rm Poi}(\mu)$ with the parameter $\mu$.
Hence $\mathbb{P}(\kappa=n)\leq  \mathbb{P}(-\frac{S_n}{n}\leq 0)$
and this probability decays exponentially to zero by
Cram\'er-Chernoff Theorem.
This completes the proof.
\end{proof}
We are ready to state the first corollary.
\begin{cor}\label{corbounds}
    Assume that
    claim events $\{ U_i \}_{i \in \mathbb N}$ of the risk process \eqref{CLprocess} are light-tailed and
    that
    there exists $R>0$ such that
    \begin{equation*}
    a(\mathbb{E}\left\{\mathbb{E}\left[e^{RU}\right]\right\}^\kappa-1)=cR.
    \end{equation*}
    Then, there exist positive constants $C_-$ and $C_+$ such that
    \[C_-e^{-Ru} \leq \psi(u) \leq C_+ e^{-Ru}.\] 
    \end{cor}
    \begin{proof}
        By Lemma \ref{kappalighttailed} it follows that the generic $\tilde{U}$ defined in \eqref{tildeU} is light-tailed and
        that
            \begin{equation}\label{adjustmentcoeff}
    a(\mathbb{E}\left[e^{R\tilde{U}}\right]-1)=cR.
    \end{equation}
        Under our assumptions, we have by Theorem 5.4.1 on p. 170 of \cite{Rolski.1999} that there exist constants $\tilde{C}_-$, $C_+$ such that the ruin probability of the clustered surplus process $\tilde \psi(u)$ satisfies \[\tilde{C}_-e^{-Ru} \leq \tilde \psi(u) \leq C_+ e^{-Ru},\] for all $u \geq 0$. Having this, we get that \[C_+e^{-Ru} \geq \tilde \psi(u) \geq \psi(u) \geq \tilde \psi(u+cl_1) \geq \tilde{C}_- e^{-Rcl_1} e^{-Ru} =: C_- e^{-Ru}.\]
    \end{proof}
\begin{remark}
\rm This behaviour implies log-convergence of the ruin probabilities is the Hawkes model, as it was already derived by \cite{Karabash.2015}.
\end{remark}

We identify the asymptotics of the ruin probability also in the case when the generic claim size is heavy-tailed.
We introduce now the appropriate class of distributions that we will work with.
Let $F_{\tilde{U}}$ be distribution of $\tilde{U}$ defined in \eqref{tildeU}.
We denote
\[
F^s_{\tilde U}(u):=\frac{1}{\mathbb{E}\left[\tilde{U}_1\right]} \int_0^u (1-F_{\tilde{U}}(y)) \mathrm{d}y.
\]
We say that a distribution $F$ with unbounded support is strongly subexponential (writing $F\in \mathcal{S}^*$) if
\[\lim_{u\to+\infty}
\frac{\int_0^u(1-F(u-y))(1-F(y))\mathrm{d}y}{2\int_0^\infty (1-F(y))\mathrm{d}y}=1.\]
It is known (see \cite{kluppelberg1988}) that any distribution from class $\mathcal{S}^*$ is subexponential (writing $F\in \mathcal{S}$), that is, that
\[\lim_{u\to+\infty}
\frac{1-F^{*2}(u)}{2(1-F(u))}=1.\]
Classical examples of distributions from the class $\mathcal{S}^*$ are Pareto, log-normal and
Weibull with parameter from $(0,1)$.

\begin{cor}\label{pierwszecor}
Assume that the integrated tail distribution $F^s_{\tilde U}$ of the generic clustered claim size $\tilde U$ is subexponential, that is, that $F^s_{\tilde U}\in \mathcal{S}$.
Then,  \[ \lim_{u \to \infty} \frac{\psi(u)}{1-F^s_{\tilde U}(u)}=\frac{\rho}{1-\rho},\] where $$\rho: = \frac{a \mathbb{E}\left[\tilde U_1\right]}{c}<1$$ by \eqref{netprofit}.
\end{cor}
\begin{proof}
    Since $F^s_{\tilde U}$ is subexponential, the ruin probability of the Cram\'er-Lundberg model satisfies by Theorem 5.4.3 on p. 175 of \cite{Rolski.1999} \[ \lim_{u \to \infty} \frac{\tilde \psi(u)}{1-F^s_{\tilde U}(u)}=\frac{\rho}{1-\rho} =  \lim_{u \to \infty} \frac{\tilde \psi(u+cl_1)}{1-F^s_{\tilde U}(u+cl_1)}.\] Further, we have that $\lim_{u \to \infty} \frac{1-F^s_{\tilde U}(u+y)}{1-F^s_{\tilde U}(u)}=1$ for all $y \in \mathbb R$ since any subexponential distribution is long-tailed (see e.g. Lemma 3.2 on p. 40 of \cite{foss2011introduction}). Consequently, \begin{multline*}
        \frac{\rho}{1-\rho} = \lim_{u \to \infty} \frac{\tilde \psi(u)}{1-F^s_{\tilde U}(u)} \geq  \limsup_{u \to \infty} \frac{\psi(u)}{1-F^s_{\tilde U}(u)} \geq \liminf_{u \to \infty} \frac{\psi(u)}{1-F^s_{\tilde U}(u)} \\
        \geq \lim_{u \to \infty} \frac{1-F^s_{\tilde U}(u+cl_1)}{1-F^s_{\tilde U}(u)} \frac{\tilde \psi(u+cl_1)}{1-F^s_{\tilde U}(u+cl_1)} =\frac{\rho}{1-\rho}
    \end{multline*}
which completes the proof.
\end{proof}

It is more valuable to derive the asymptotics of the ruin probability $\psi(u)$ in terms of the original
distribution of the claim sizes $F_U$, under assumption that $F_U$ is strongly subexponential
(hence heavy-tailed by Lemma 3.2 on p. 40 of \cite{foss2011introduction}).

\begin{cor}\label{aymptoticssubexponential}
If $F_U\in \mathcal{S}^*$ then
\begin{equation}\label{subasymptotics2}
\lim_{u \to \infty} \frac{\psi(u)}{1-F^s_{U}(u)}=\frac{\rho}{1-\rho}.\end{equation}
\end{cor}
\begin{proof}
Recall that $\tilde{U}=\sum_{k=1}^\kappa U_k$ for a generic cluster size $\kappa$ and by Lemma \ref{kappalighttailed}
$\kappa$ is light-tailed.
Now the statement follows from
Theorem 1 of \cite{MR2759165}, \eqref{meankappa} and Corollary \ref{pierwszecor}
since
\[\lim_{u \to \infty} \frac{1-F_{\tilde U}(u)}{1-F_{U}(u)}=\frac{1}{1-\mu}\]
and hence
\[\lim_{u \to \infty} \frac{1-F^s_{\tilde U}(u)}{1-F^s_{U}(u)}=1.\]
\end{proof}

\begin{remark}
We recover Proposition 13 of \cite{MR3130449}.
Our proof is different though not
requiring tedious checking of Assumptions 1 of
\cite{MR3130449}.
\end{remark}

Although the results of Theorem \ref{thm:general_case} are similar to the findings of \cite{Albrecher.2006}, it faces two drawbacks. The first problem is that the asymptotic behaviour depends highly on the distribution of the sum of all claim sizes occurring in a single cluster, something that might be hard to handle. The second weakness is that we were not able to show if there are conditions under which $\lim_{u\to \infty} e^{Ru} \psi(u)$ converges. To achieve this, we restrict ourselves to the Markovian version of the Hawkes process.

\section{Cram\'er-Lundberg asymptotics for Markovian Hawkes model} \label{sec:Markov}
\subsection{The Markovian Hawkes model}
The intensity of a Hawkes process is generally not Markovian. To resolve this, we have to choose the specific decay function $h(t,y)= e^{-\beta t}y$, for some positive decay parameter $\beta$. Further, for some positive constants $\lambda > a >0$, we define the Markovian (marked) Hawkes process $\{ N_t \}_{t \geq 0}$ by its intensity process
\begin{equation}\label{lambdamarkov}
 \lambda_t := a + (\lambda -a)e^{-\beta t} + \sum_{i=1}^{N_t} Y_i  e^{-\beta(t-T_i)}.\end{equation}
 Here, the random variables $\{ Y_i\}_{i \in \mathbb N}$ are assumed to be i.i.d. copies of a positive random variable $Y$ with cumulative distribution function $F_Y$ and finite expectation.  In contrast to the previous part, we allow for different initial values $\lambda$ instead of restricting ourselves to the case $\lambda_0=a$.
 This process is well-defined if \[ \int_0^\infty \mathbb E[h(t,Y)] \, \mathrm{d}t = \mathbb E[Y] \int_0^\infty e^{-\beta t} \, \mathrm{d}t = \frac{\mathbb E[Y]}{\beta} < 1,\] which gives us the restriction that
 \begin{equation}\label{dodzal}\beta > \mathbb E[Y],\end{equation}
 whereas the integrability condition \[ \int_0^\infty t\,\mathbb E[h(t,Y)] \, \mathrm{d}t < +\infty,\] is always satisfied.
This process is called Markovian, since the intensity process $\{ \lambda_t \}_{t \geq 0}$ is a piecewise deterministic Markov process (PDMP) with extended generator \[ \mathcal{A}^\lambda f(\lambda) = \beta (a-\lambda) \frac{\partial}{\partial \lambda}f(x,\lambda,t)  + \lambda \int_0^\infty  f(\lambda+y) \, F_Y(\mathrm{d}y) - \lambda f(x,\lambda,t);\]
see \cite{Davis.1984} for more details on theory of PDMPs.
We recall that for any Markov process $\{Z_t\}_{t\geq 0}$ we say that $\mathcal{A}$ is its extended generator and $D(\mathcal{A})$ is a domain of this generaotor if
if $f(Z_t)-f(Z_0) - \int_0^t g(Z_s) \, \mathrm{d}s$ is a zero mean local martingale with respect to its natural filtration for
$f \in \mathcal{D}(\mathcal{A})$
and some function $g$. We then write $g=\mathcal{A}f$.

Further, the multivariate process $\{ (X_t, \lambda_t, t)\}_{ t\geq 0}$ is also a piecewise deterministic Markov process  with full generator
\begin{multline}\label{generator}
    \mathcal{A}f(x,\lambda,t) = c \frac{\partial}{\partial x}f(x,\lambda, t) + \beta (a-\lambda) \frac{\partial}{\partial \lambda}f(x,\lambda,t) + \frac{\partial}{\partial t}f(x,\lambda,t)  \\ +\lambda \int_0^\infty \int_0^\infty f(x-u,\lambda+y,t) \, F_U(\mathrm{d}u) \, F_Y(\mathrm{d}y) - \lambda f(x,\lambda,t).
\end{multline}
Since we have now two different initial values, one for the surplus process and one for the intensity, we write $\mathbb{P}_{(u,\lambda)}$ for the measure $\mathbb{P}$ under condition that $X_0=u$ and $\lambda_0=\lambda$ for $u\geq 0$ and $\lambda>a$, and $\mathbb{E}_{(u,\lambda)}\left[\cdot\right]$ for the corresponding expectation. If a stochastic object $Z$ is independent of the initial values, we will omit them and write $\mathbb E [Z]$ instead of $\mathbb{E}_{(u,\lambda)}\left[Z\right] $. Having this, we have by \cite[p. 449]{Rolski.1999}, that the domain $\mathcal{D}(\mathcal{A})$ of this generator $\mathcal{A}$ consists of all functions $f:\mathbb{R}^3 \to \mathbb{R} $ such that the mapping $t \to f(x+ct, (\lambda-a)e^{-\beta t} +a, s+ t)$ is absolutely continuous for almost all $(x,\lambda,s)$ and \[ \mathbb{E}_{(u,\lambda)} \left[ \sum_{k=1}^{N_t} \left \vert f(X_{T_k}, \lambda_{T_k}, T_k) - f(X_{T_k-}, \lambda_{T_k-}, T_k)\right\vert\right] < +\infty , \] for all $t \geq 0$.

We start from identifying a suitable net profit condition, that is, the condition under which \eqref{net} holds.
In fact, we will identify the limiting value $\lim_{t \to \infty} \frac{\mathbb{E}_{(u,\lambda)} \left[X_t\right]}{t}$
and assume that this limit is strictly positive which gives \eqref{net}.
\begin{lem}\label{lem:mean_surplus}
    The surplus process satisfies $\lim_{t \to \infty} \frac{\mathbb{E}_{(u,\lambda)} \left[X_t\right]}{t}=c- \frac{a\beta \mathbb{E}[U]}{\beta - \mathbb{E}[Y]}$.
    \end{lem}
    \begin{proof}
        The function $f(x,\lambda,t) = x$ is in the domain of the generator and satisfies $\mathcal{A}f(x,\lambda,t) = c - \lambda \mathbb{E}[U]$. Therefore, we have that $ \mathbb{E}_{(u,\lambda)}\left[X_t\right] = u + ct - \mathbb{E}[U] \int_0^t \mathbb{E}_{\lambda}[\lambda_s] \, \mathrm{d}s$. By \cite{Cui.2020}, we have that $\mathbb{E}_{\lambda} [\lambda_s] = \frac{\beta a}{\beta-\mathbb{E}[Y]} + \left( \lambda - \frac{\beta a}{\beta - \mathbb{E}[Y]}\right) e^{- s(\beta - \mathbb{E}[Y])}$. Consequently, $\int_0^t \mathbb{E}_{\lambda}[\lambda_s] \, \mathrm{d}s = \frac{\beta a}{\beta-\mathbb{E}[Y]} t + o(t)$. Using this, we get \[ \lim_{t \to \infty}  \frac{\mathbb{E}_{(u,\lambda)} \left[X_t\right]}{t} = c -\mathbb{E}[U]\frac{\beta a}{\beta-\mathbb{E}[Y]}  \]
      which completes the proof.
    \end{proof}

By this result, we propose the following net profit condition.
\begin{assumption}\label{assumption1}
    From now on we assume that the net profit condition \[c > \frac{a\beta \mathbb{E}[U]}{\beta - \mathbb{E}[Y]}\] holds. Further, we assume that there exist some positive $s_Y$ such that the moment generating function
    \[M_Y(s):=\mathbb{E}\left[e^{sY}\right] < +\infty \, \forall s< s_Y, \quad\text{and}\quad \lim_{s \to s_Y} M_Y(s) = +\infty.\]
\end{assumption}

\subsection{Steps of the proof of Cram\'er-Lundberg asymptotics}
Our main goal of this section is to prove (under additional assumptions) for the Hawkes process with the intensity \eqref{lambdamarkov} so-called Cram\'er Lundberg asymptotics, that is,
that there exists an adjustment coefficient $R$ solving Lundberg equation
formulated later (in \eqref{adjustment2})
and constant $C^\lambda>0$, depending on $\lambda$, such that
\[ \lim_{u \to \infty} \psi(u,\lambda)e^{Ru} = C^\lambda;\]
see Theorem \ref{cramerasymptotics}.

To prove this statement we divide the whole proof into the following steps:
\begin{enumerate}
\item in next subsection we prove in Theorem \ref{thm:posrecc} that the intensity process $\{\lambda_t\}_{t\geq 0}$ is positive Harris recurrent
and, in Theorem \ref{thm:light_tailed_recurrence},  that the corresponding recurrence times $\{\phi_i\}_{i \in \mathbb N}$ are light-tailed;
\item  later, in Definition \ref{defalpha}, we introduce the functions $\alpha(r)$ and $\theta(r)$
as special solutions of equations \eqref{firsteq}-\eqref{maineq};
\item the adjustment coefficient $R>0$ is defined as a maximal solution of the Lundberg equation
\[\theta(R)=0;\]
\item we introduce, in Definition \ref{newmeasuredef} using Theorems \ref{localmartingale1} and \ref{localmartingale2}, the new exponential measure
\[\left.\frac{\mathrm{d}\mathbb{Q}^{(R)}}{\mathrm{d}\mathbb{P}}\right|_{\mathcal{F}_{t}}=M_t:=\exp\left(Ru+\alpha(R)\lambda) \exp(-RX_t-\alpha(R)\lambda_t\right); \] 
\item in Lemma \ref{driftpostive} we prove that under the new measure $\mathbb{Q}^{(R)}$ ruin occurs almost surely and the intensity $\{ \lambda_t\}_{t \geq 0}$ remains positive Harris recurrent (see Lemma \ref{newmeasurerecc});
\item we introduce
\[B(\mathrm{d}x) := e^{Rx} \tilde{B}(\mathrm{d}x)\quad\text{for}\quad  \tilde{B}(x):= e^{Rx}\mathbb{P}_{(u,\lambda)} \left[ \phi_1 < +\infty, u-X_{\phi_1}  \leq x\right]\]
and
\[p(u,x):= \mathbb{P}_{(u,\lambda)}\left[\tau \leq \phi_1, \, \left\vert \phi_1 < +\infty, X_{\phi_1} =u-x \right.\right]\]
for the ruin time $\tau$ defined in \eqref{ruintime} and first reccurence to $\lambda$ epoch $\phi_1$ of $\{\lambda_t\}_{t\geq 0}$ such that
process $\{X_t\}_{t \geq 0}$ is getting below initial starting position $u$; now using classical one-step analysis we can prove that
\[Z(u)=\psi(u) e^{Ru}\]
satisfies the following renewal equation (see \eqref{renewalequation2})
\begin{equation*}
    Z(u)= \int_0^u Z(u-x) (1-p(u,x)) B(\mathrm{d}x) + z(u),
\end{equation*}
where
\[z(u)= e^{Ru}\mathbb{P}_{(u,\lambda)}\left[\tau\leq \phi_1, \tau < +\infty \right];\]
\item to show that $\lim_{u\to+\infty} Z(u)$ exists and is finite, in the last step of the proof we use
Theorem 2 of \cite{Schmidli.1997} and the first step of the proof to prove that $z(u)$ and $\int_0^u p(u,x) B(\mathrm{d}x)$ are directly Riemann integrable.
\end{enumerate}

We start from the analysis of the behaviour of the Markovian intensity process $\{\lambda_t\}_{t\geq 0}$.

\subsection{Harris recurrence of the intensity process}\label{sec:recurrence}
In this section, we investigate the behaviour of the intensity of the Markovian marked Hawkes process. Our goal is to show that the intensity process is positive Harris recurrent (see Theorem \ref{thm:posrecc})
and that the corresponding recurrence times are light-tailed (see Theorem \ref{thm:light_tailed_recurrence}), both properties are needed to determine the asymptotic behaviour of the ruin probability.

To define these properties properly, we consider a right-continuous, time-homogeneous, strong Markov process $\{ Z_t \}_{t \geq 0}$ on $(\mathcal{E},\mathcal{B}(\mathcal{E}))$. Here, $\mathcal{E}$ denotes a locally compact, separable metric space and $\mathcal{B}(\mathcal{E})$ its Borel $\sigma$-algebra.
\begin{definition}
    The process $\{ Z_t \}_{t \geq 0}$ is called Harris recurrent if there exists a $\sigma$-finite measure $\varphi$ on $\mathcal{B}(\mathcal{E})$ such that $\varphi(B)>0 \Rightarrow \mathbb{P}_z\left[ \eta_B=\infty\right]=1$ for all initial values $z$ and $B\in\mathcal{B}(\mathcal{E})$, where
    \begin{equation}\label{eta}\eta_B:= \int_0^\infty I_{\left\lbrace Z_t \in B\right\rbrace } \mathrm{d}t\end{equation}
    denotes the occupation measure of the process $\{ Z_t \}_{t \geq 0}$. It is called positive Harris recurrent if it is Harris recurrent with finite invariant measure $\pi$.
\end{definition}

To show that our intensity process satisfies this property, we need the following definitions of a continuous component and a $T$-process as in Section 3.2 of \cite[pp. 495-496]{Meyn.1993}.

\begin{definition}
    Let $\{ Z_t \}_{t \geq 0}$ be our strong Markov process and $\sigma_1, \sigma_2, \ldots$ an i.i.d. sequence of positive random variables with distribution $F$ and independent of $\{ Z_t \}_{t \geq 0}$. Then, we define the embedded Markov chain $Y_n:=Z_{\sigma_1+ \cdots + \sigma_n}$ with one-step transition probability $K_F(z,A):= \int_0^\infty \mathbb{P}_z\left[Z_t \in A \right] F(\mathrm{d}t)$ for all $A \in \mathcal{B}(\mathcal{E})$. A kernel $T: (\mathcal{E},\mathcal{B}(\mathcal{E}))\to \mathbb{R}_+$ is called a continuous component of $K_F$ if $K_F(x,A) \geq T(x,A)$ for all $x$ and $A$, and for fixed $A \in \mathcal{B}(\mathcal{E})$, the function $T(\cdot, A)$ is lower semi-continuous. We say $T$ is non-trivial if, for all $x\in \mathcal{E}$, we have that $T(x,\mathcal{E})>0$.
\end{definition}
A special case of such an embedded Markov chain is the resolvent chain, whose transition kernel is given by $R_\gamma (x,A) := \int_0^\infty \mathbb{P}_x\left[Z_t \in A \right] e^{-\gamma t} \mathrm{d}t$, i.e. where $\sigma_1 \sim {\rm Exp}(\gamma)$ for the exponential random variable ${\rm Exp}(\gamma)$ with the parameter $\gamma>0$.
\begin{definition}
    The process $\{ Z_t \}_{t \geq 0}$ is called $T$-process if there is a probability distribution $F$ such that $K_F$ admits a non-trivial continuous component $T$.
\end{definition}

\begin{definition}
    Recall that $\eta_B$ defined in \eqref{eta} is the occupation measure. Let $\varphi$ be a $\sigma$-finite measure. If $\varphi(B)>0 \Rightarrow \mathbb{E}_z\left[ \eta_B\right] >0$ for all initial values $z$ and all $B \in \mathcal{B}(\mathcal E)$, then $\{ Z_t \}_{t \geq 0}$ is called $\varphi$-irreducible.
\end{definition}

\begin{definition}
    A process $\{ Z_t \}_{t \geq 0}$ is called bounded in probability on average if for every initial value $z$ and $\varepsilon>0$ there is a compact set $K$ such that \[ \liminf_{t \to \infty} \frac{1}{t} \int_0^t \mathbb{P}_z\left[Z_s \in K \right]\mathrm{d}s \geq 1-\varepsilon.\]
\end{definition}
These three properties are related to positive Harris recurrence by Theorem 3.2 of \cite{Meyn.1993}, which states the following.

\begin{thm}\label{thm:Harris_meyn}
    Suppose that $\{ Z_t \}_{t \geq 0}$ is a $\varphi$-irreducible $T$-process. Then $\{ Z_t \}_{t \geq 0}$ is positive Harris recurrent if and only if it is bounded in probability on average.
\end{thm}

Now, we want to show that our process $\{\lambda_t\}_{t\geq0}$ satisfies all conditions of Theorem \ref{thm:Harris_meyn}. It is a time-homogeneous strong Markov process with right-continuous paths defined on $(\mathbb R, \mathcal{B}(\mathbb R))$ . The space $\mathbb{R}$ is a locally compact and separable metric space. The next point is to show that the process is a $T$-process.
\begin{lem}
    The intensity process $\{\lambda_t\}_{t\geq0}$ is a $T$-process.
    \end{lem}
    \begin{proof}
        For this, we show that the resolvent kernel $R_1(\lambda,A) = \int_0^\infty \mathbb{P}_\lambda\left[\lambda_t \in A\right] e^{-t} \mathrm{d}t$ is continuous in $\lambda$, for every fixed $A\in \mathcal{B}(\mathbb{R})$. For this, observe that by Tonelli's theorem, we can interchange expectation and integration to get that \[ R_1 (\lambda,A) = \mathbb{E}_\lambda \left[ \int_0^\infty I_A(\lambda_t) e^{-t} \mathrm{d}t\right]=:V(\lambda).\]
        Further, by Theorem 31.9 of \cite{Davis.1993} we know that \[V(\lambda)= \tilde V(\lambda) := \mathbb{E}_\lambda \left[ \int_0^\infty I_A(\tilde \lambda_t) \mathrm{d}t\right], \] where $\{ \tilde \lambda_t\}_{t\geq 0}$ denotes the process $\{\lambda_t\}_{t\geq0}$ but killed with constant rate $1$.  The function $\tilde V(\lambda)$ is bounded by $1$, and the function $l(x):= I_A(x)$ is measurable and integrable. Therefore, we have by Theorem 32.2 of \cite{Davis.1993}, that $\tilde V(\lambda)$ is absolutely continuous. The kernel $R_1$ is non-trivial since $R_1(\lambda, \mathbb{R}_+) =1$ for all $\lambda \in \mathbb{R}_+$. Therefore, $R_1$ serves as non-trivial component.
    \end{proof}

To show that the intensity process is $\varphi$-irreducible, we have to identify a suitable $\sigma$-finite measure $\varphi$. For this, we will show in the first step that the Markovian Hawkes intensity converges in distribution to a stationary probability distribution $\nu$.
To do so, we introduce the process $\mu_t= a+ \sum_{i=1}^{N_t} Y_i e^{- \beta (t-T_i)}$, i.e. our Hawkes intensity with initial condition $\lambda_0=a$. By \cite[p. 133]{Bremaud.2002}, the distribution of $\mu_t$ converges weakly against a stationary distribution as $t \to \infty$.
If we can show that, independent of the initial intensity $\lambda$, the process $\{ \lambda_t-\mu_t\}_{t \geq 0}$ converges in probability to $0$. Then, we have by Slutsky's theorem that $\lambda_t = \lambda_t - \mu_t + \mu_t $ converges in distribution to $\nu$ too.
\begin{lem}
  Let $\lambda > a$ be arbitrary but fixed. Then $\lim_{t \to \infty} \mathbb{P}_{\lambda} \left[\left\vert\lambda_t-\mu_t\right\vert > \varepsilon\right]=0 $ for all $\varepsilon>0$.
  \end{lem}
  \begin{proof}
      To show this, we want to use Markov's inequality. Therefore, we are interested in the behaviour of $\mathbb{E}[ \left\vert \lambda_t-\mu_t \right\vert ]$ as $t \to \infty$. For this, we take a look at the change of the intensity by the increase of the initial value by $\lambda-a$.
      For fixed $\lambda$, we can decompose the corresponding counting process $\{N_t\}_{t\geq 0}$ as \[N_t= N^\mu_t + M_t,\] where $ M_t$ counts all jumps due to the initial increase by $\lambda-a$ and $N^\mu_t$ is a marked Hawkes process with intensity $\mu_t$.

      In the time interval $(0,\infty)$, the initial increase by $\lambda-a$ will cause $$Z:={\rm Poi}\left((\lambda-a)\int_0^\infty e^{-\beta t} \mathrm{d}t\right)= {\rm Poi} \left(\frac{\lambda-a}{\beta}\right)$$ jumps of the marked Hawkes process. We will call these jumps 'children'. Each child increases the intensity by a generic $\tilde Y$, hence triggers ${\rm Poi}(\frac{\tilde Y}{\beta})$ additional jumps, i.e. 'grandchildren'. These grandchildren cause new jumps again, so we get a branching structure. We call the collection of all jumps caused by $\lambda-a$ offspring. Since $\mathbb{E}\left[ \int_0^\infty Y e^{-\beta t} \mathrm{d}t \right] = \frac{\mathbb{E}\left[Y\right]}{\beta}< 1$, we have that the number of jumps in such a cluster is an integrable random variable, see \cite{Basrak.2019}. The random variable $M_t$ corresponds to the number of offspring due to the increase by $\lambda-a$, which appeared up to time $t$ and converges almost surely to the integrable random variable $M_\infty$, which corresponds to the size of the cluster caused by the additional initial intensity.

      By this, we have that \[\left\vert \lambda_t - \mu_t \right\vert = (\lambda -a)e^{-\beta t} + \sum_{i=1}^{M_t} \tilde Y_i e^{-\beta (t- T^M_i)} \leq (\lambda -a)e^{-\beta t} + \sum_{i=1}^{M_\infty} \tilde Y_i, \]  where $\{ T_i^M\}_{i \leq M_\infty}$ are the jump times of the counting process $\{M_t\}_{t\geq 0}$. The upper bound is integrable. Hence, by dominated convergence, we have that \begin{multline*}
          \limsup_{t \to \infty}\mathbb{E}_\lambda \left[\left\vert \lambda_t-\mu_t\right\vert\right] \leq \mathbb{E}_\lambda \left[ \limsup_{t \to \infty}\;(\lambda-a)e^{-\beta t} + \sum_{i=1}^{M_\infty} \tilde Y_i e^{-\beta (t-T^M_i)}\right] \\= \mathbb{E}_\lambda \left[ \sum_{i=1}^{M_\infty} \limsup_{t \to \infty} \tilde Y_i e^{-\beta (t-T^M_i)}\right]=0.
      \end{multline*}
The statement follows from Markov's inequality.
  \end{proof}

\begin{thm}
    The Markovian Hawkes intensity converges in distribution to the stationary distribution $\nu$.\end{thm}
    \begin{proof}
        Let $\{\lambda_t\}_{t \geq 0}$ and $\{\mu_t\}_{t \geq 0}$ be as before. By the previous lemma we have that $\lambda_t-\mu_t \to 0$ in probability, and by \cite{Bremaud.2002}, we get that $\mu_t \to \nu$ in distribution. Slutsky's theorem gives us that \[\lambda_t = \lambda_t-\mu_t +\mu_t \convdistr 0+ Z,\] where
        $Z \sim
        \nu$.
    \end{proof}

 To show that the intensity process is $\nu$-irreducible, we still need some smoothness of the stationary distribution $\nu$.

\begin{lem}\label{lem:absolutely_continuous}
    The stationary distribution $\nu$ is absolutely continuous with respect to the Lebesgue measure.\end{lem}
    \begin{proof}
        This is due to Proposition 1.9 of \cite{Lopker.2013}.
    \end{proof}
Let $B \in \mathcal{B}(\mathbb R)$ and define $\nu_t(B):= \frac{1}{t}\int_0^t \mathbb{P}_\lambda\left[\lambda_s \in B\right]\mathrm{d}s$.
\begin{lem}\label{lem:convergence_nut}
    The family of measures $\{ \nu_t\}_{t \geq 0}$
    converges weakly to the stationary measure $\nu$ as $t\to \infty$.\end{lem}
    \begin{proof}
        Since $\lambda_t$ converges in distribution to the absolutely continuous measure $\nu$, we have that for all open sets $U$ that $\liminf_{t \to \infty} \mathbb{P}_\lambda\left[\lambda_t \in U\right]\geq \nu(U).$ Consequently, for all $\varepsilon>0$ there is a $T>0$ such that for all $t \geq T$, it holds that $\mathbb{P}_\lambda\left[\lambda_t \in U\right]\geq \nu(U)-\varepsilon.$ Therefore, \begin{align*}
            &\liminf_{t \to \infty} \frac{1}{t}\int_0^t\mathbb{P}_\lambda\left[\lambda_s \in B\right]\mathrm{d}s \\ &\quad\geq\liminf_{t \to \infty} \frac{1}{t}\left(\int_0^T\mathbb{P}_\lambda\left[\lambda_s \in B\right] \mathrm{d}s + (t-T) (\nu(U)-\varepsilon)\right) \\ &\quad= \nu(U)-\varepsilon.
        \end{align*}
        If we let $\varepsilon$ tend to $0$, we see that $\liminf_{t\to\infty}\nu_t(B) \geq \nu(B)$ for all open sets $B$. By the Portmanteau theorem, this implies that $\{ \nu_t \}_{t \geq 0}$ converges weakly to the stationary measure $\nu$.
    \end{proof}
Now we are ready to show that our process is $\nu$-irreducible.
\begin{lem}
    The Markovian intensity process $\{\lambda_t\}_{t\geq 0}$ is $\nu$-irreducible.\end{lem}
    \begin{proof}
        Let $B$ be measurable with $\nu(B)>0$ and $\nu_t$ as in Lemma \ref{lem:convergence_nut}. By the absolute continuity of the measure $\nu$, we have that $B$ is a continuity set of $\nu$. Hence, $\lim_{t \to \infty}\nu_t(B) = \nu(B) >0$. Consequently, we have that $\mathbb{E}_\lambda\left[\eta_b\right]= \int_0^\infty \mathbb{P}_\lambda \left[ \lambda_s \in B\right] \mathrm{d}s =\infty>0$.
    \end{proof}

\begin{lem}
    The Markovian intensity process $\{ \lambda_t\}_{t \geq 0}$ is bounded in probability on average.\end{lem}
    \begin{proof}
        Let $\varepsilon>0$ be arbitrary and $K_\varepsilon\subset \mathbb{R}_+$ be compact such that $\nu(K) \geq 1-\varepsilon$. Then we have by Lemma \ref{lem:absolutely_continuous} and Lemma \ref{lem:convergence_nut}, that \[ \liminf_{t \to \infty} \frac{1}{t}\int_0^t \mathbb{P}_\lambda \left[ \lambda_s \in K\right] \mathrm{d}s = \nu(K) \geq 1-\varepsilon.\]
    \end{proof}

\begin{thm}\label{thm:posrecc}
    The Markovian intensity process $\{ \lambda_t\}_{t \geq 0}$ is positive Harris recurrent.\end{thm}
    \begin{proof}
        This follows directly from Theorem \ref{thm:Harris_meyn} and the previous lemmas.
    \end{proof}

This gives us that our intensity process visits all sets with $\nu(B)>0$ infinitely often, but we still have to
check 
that the stationary distribution has support $(a,\infty)$.
First, we will show that the support is unbounded from above.
\begin{lem}
    The support of $\nu$ is unbounded from above.\end{lem}
    \begin{proof}
        Assume there is a finite bound $b>0$ such that $ \nu((a,b)) =1$. We use that there is a stationary version of our intensity process, and we will denote it by $\lambda'_t= a + (\lambda'_0 - a) e^{-\beta t} + \sum_{i=1}^{N'_t} Y_i e^{- \beta (t-T'_i)}$, where $\lambda'_0 \sim \nu$. Choose $\delta>0$ arbitrary. Then, we have \begin{multline*}
            0= \nu((b,\infty)) = \mathbb{P}_\nu\left[\lambda'_\delta >b \right] \geq\mathbb{P}_\nu\left[\lambda'_\delta >b, N'_\delta =1 \right] \\ =\mathbb{P}_\nu\left[\lambda'_\delta >b \left\vert N'_\delta =1 \right.\right] \mathbb{P}_\nu \left[N'_\delta =1 \right] \geq \mathbb{P}\left[Y>be^{\delta \beta}\right] \mathbb{P}_\nu\left[N'_\delta=1 \right] >0,
        \end{multline*}  which is a contradiction. Since the support of the shock events $Y$ is unbounded, we consequently have that the support of the stationary distribution is unbounded.
    \end{proof}
\begin{lem}\label{lem:support_is_interval}
    The support of $\nu$ is an open set of the form $(b, \infty)$ for some $b \geq a$.
    \end{lem}
    \begin{proof}
        We already know that the support is unbounded. Assume that the support of the stationary distribution is not an open interval. Since $\nu$ is absolutely continuous with respect to the Lebesgue measure, we have that there exists an interval $(c,d) \subset (a,\infty)$ such that $\nu((c,d)) =0$, $\nu((a,c))>0$, and $\nu(d,\infty)>0$. Now, we assume that this interval is maximal. In particular, we want that for all $\varepsilon>0$ $\nu((d,d+\varepsilon))>0$. Let $  -\frac{1}{\beta} \ln\left(\frac{c-a}{d-a+\varepsilon}\right)>\delta > -\frac{1}{\beta} \ln\left(\frac{d-a}{d-a+\varepsilon}\right)$ deterministic. Then we have that if $\lambda'_0 \in (d, d+\varepsilon) $ and no jump occurs between time $0$ and $\delta$, $\lambda'_\delta \in (c, d) $, which is a contradiction. Writing this down, we get that
        \begin{align*}
            & \nu((c,d)) = \mathbb{P}_\nu \left[\lambda'_{\delta} \in (c,d)\right]\geq \mathbb{P}_\nu \left[\lambda'_{\delta} \in (c,d), \lambda'_0 \in (d,d+\varepsilon) \right] \\&\quad
            = \mathbb{P}_\nu \left[\lambda'_{\delta} \in (c,d)\left\vert  \lambda'_0 \in (d,d+\varepsilon)\right. \right]\mathbb{P}_\nu \left[ \lambda'_0 \in (d,d+\varepsilon) \right] \\&\quad \geq \mathbb{P}_\nu \left[ N'_\delta= 0 \right] \nu((d,d+\varepsilon)) >0.
        \end{align*}
        This is a contradiction. Hence, the support of $\nu$ is an interval.
    \end{proof}

\begin{thm}
    The support of $\nu$ is $(a,\infty)$.
\end{thm}
    \begin{proof}
        This proof is similar to the proof of Lemma \ref{lem:support_is_interval}. Assume that the support of $\nu$ is not $(a, \infty)$. Then there exists a $\varepsilon>0$ such that the support is $(a+\varepsilon, \infty)$ and $\nu((a+\varepsilon, a+2\varepsilon))>0$. Let $\delta >-\frac{1}{\beta}\ln(\frac{1}{2})$. Then again, if $\lambda'_0 \in (a+\varepsilon, a+2\varepsilon)$ and $N'_\delta=0$, which both happen with positive probability, then $\lambda'_\delta <a+\varepsilon$. This contradicts the assumption that $\nu((a,a+\varepsilon))=0$.
    \end{proof}

This gives us that our intensity process visits every open interval in $(a, \infty)$ infinitely often. Since it decays only in a continuous way via its exponentially decaying drift, we even have that the process $\{ \lambda_t \}_{t\geq 0}$ visits every single point $\lambda \in (a,\infty)$ infinitely often with probability $1$.

\begin{thm}\label{thm:light_tailed_recurrence}
Let $\lambda >a$ be arbitrary and $S_1^\lambda$ the first positive time point such that $\lambda_{S^ \lambda_i} =\lambda$. Then, there exists a $r>0$ such that $\mathbb{E}_\lambda\left[e^{r S^\lambda_1}\right] < +\infty.$
\end{thm}
\begin{proof}
The Markovian Hawkes process satisfies Scenario 1.1 of \cite{Borovkov.2008} and $\beta (a-\lambda) \neq 0$ for all $\lambda >a$. Therefore, we have that, under the stationary distribution, the number of continuous crossings of our process through $\lambda$ has intensity $\mu_c(\lambda):=  \beta (\lambda -a) p(\lambda)>0$. Here, $p(\lambda)$ denotes the density of the stationary distribution $\nu$. Consequently, we have \[\mathbb{P}_\nu \left[ S_1^\lambda >t\right] = \exp\left(-\int_0^t \beta (\lambda-a) p(\lambda) \, \mathrm{d}s\right) = \exp(-t\beta (\lambda-a) p(\lambda)).\] This implies that $\mathbb{E}_\nu \left[ e^{r S_1^\lambda}\right] < +\infty$ for all $r < \beta (\lambda-a) p(\lambda)$.  \\
Using the fact that $\nu$ is absolutely continuous with respect to the Lebesgue measure, we have that \[\mathbb{E}_\nu \left[e^{r S^\lambda_1}\right] = \int_a^\infty \mathbb{E}_x \left[e^{r S^\lambda_1}\right] p(x) \, \mathrm{d}x < +\infty ,\] which gives us that
$\mathbb{E}_x\left[e^{r S^\lambda_1}\right]<+\infty$ for Lebesgue almost every $x>a$.
Let now $\lambda < y $ be arbitrary and write $S^\lambda_1 \vert _{\lambda_0=y}$ for the time of the first crossing of the level $\lambda$ starting in $y$. Then, there exists a $x>y$ such that $\mathbb{E}\left[e^{r S^\lambda_1 \vert _{\lambda_0=x}}\right]= \mathbb{E}_x\left[e^{r S^\lambda_1}\right]<+\infty$. The downward movement of the intensity process $\{ \lambda_t\}_{t \geq 0}$ is continuous. Hence, if it starts in $x$ and reaches the level $\lambda$ it must cross $y$. By the strong Markov property, we can restart the process after hitting $y$ and therefore \[S^\lambda_1 \vert_{\lambda_0=x} = S^y_1 \vert_{\lambda_0=x} + S^\lambda_1 \vert_{\lambda_0=y} \geq S^\lambda_1 \vert_{\lambda_0=y}.\]
Thus, \[ \mathbb{E}_y\left[e^{r S^\lambda_1}\right] \leq \mathbb{E}_x\left[e^{r S^\lambda_1}\right] < +\infty.\]
This property holds for all $\lambda$ and $y$ as long $y>\lambda$ and $r$ is chosen suitable small, depending on the choice of $\lambda$.

Consider now $S^\lambda_1 \vert_{\lambda_0 = \lambda}$. Then, there exists an $x< \lambda$ such that $\mathbb{E}_x\left[e^{r S^\lambda_1}\right] < +\infty$. Since $\lambda >x$, we know there exists a positive $\tilde r>0$ such that $\mathbb{E}_\lambda\left[e^{\tilde r S^x_1}\right] < +\infty$.
Now, there are almost surely two possibilities. Either the intensity hits the level $x$ before it returns to $\lambda$, i.e. $S^x_1\vert_{\lambda_0 = \lambda} \leq S^\lambda_1 \vert_{\lambda_0 = \lambda}$, or it first returns to $\lambda$.
If we have $\omega \in \Omega$ such that the path of $\{ \lambda_t\}_{t \geq 0}$ hits the level $x$ before returning to $\lambda$, we can use the restart argument as before and obtain the equality $S^\lambda_1 \vert_{\lambda_0=\lambda}(\omega)  = S^x_1 \vert_{\lambda_0=\lambda}(\omega) + S^\lambda_1 \vert_{\lambda_0=x}(\omega)$. For almost every other $\omega$, we have $S^\lambda_1 \vert_{\lambda_0=\lambda}(\omega)< S^x_1 \vert_{\lambda_0=\lambda}(\omega) \leq S^x_1 \vert_{\lambda_0=\lambda}(\omega) + S^\lambda_1 \vert_{\lambda_0 =x}(\omega)$.
This give us for $q$ small enough that \[ \mathbb{E}_\lambda\left[ e^{q S^\lambda_1}\right] \leq \mathbb{E}_\lambda\left[ e^{q S^x_1}\mathbb{E}_x\left[ e^{q S^\lambda_1}\right]\right] = \mathbb{E}_\lambda\left[ e^{q S^x_1}\right]\mathbb{E}_x\left[ e^{q S^\lambda_1}\right]< +\infty.\]
This ends the proof.
\end{proof}

\subsection{Exponential change of measure}
We derive now the Cram\'er-Lundberg asymptotics under assumption that claims are light-tailed.
More precisely, we assume the following:
\begin{assumption}\label{assumption2}
    From now on, we assume that the distribution of the claim sizes $F_U$ is absolutely continuous with respect to the Lebesgue measure. Further, we assume that there exists some $s_U \in (0, \infty]$ such that the corresponding moment-generating
    \[M_U(s):=\mathbb{E}\left[e^{sU}\right]\]
    is finite for all $s<s_U$ and $\lim_{s\to s_U} M_U(s) = \infty$.
\end{assumption}
We are interested in the asymptotic behaviour of the ruin probability
\begin{equation}\label{psilambd}
\psi(u)=\psi(u,\lambda):= \mathbb{P}_{(u,\lambda)}\left( \tau < +\infty \right),\end{equation}
where
\[\tau: = \inf \left\lbrace t \geq 0:  X_t\leq 0 \,  \right\rbrace.\]
The main tool to show convergence of the ruin probability is Theorem 2 of \cite{Schmidli.1997} which gives us that the solution to the generalized renewal equation
\begin{equation}\label{renewalequation}
    Z(u)= \int_0^u Z(u-x) (1-p(u,x)) B(\mathrm{d}x) + z(u)
\end{equation}
converges as $u \to \infty$ if $B(x)$ is a probability distribution, $p(u,x) \in [0,1]$ is continuous in $u$, and both $z(u)$ and $\int_0^u p(u,x) B(\mathrm{d}x)$ are directly Riemann integrable.

The first problem that occurs in this approach is that this equation is univariate, whereas the probability of ruin $\psi(u,\lambda)$ depends on the initial values of the surplus and the intensity process. To resolve this, we use the results of Section \ref{sec:recurrence}, i.e. the intensity is Harris positive recurrent. To be precise, we exploit that it returns infinitely often to its initial value with probability $1$. This allows us to choose renewal times so that they coincide with intensity recurrence times.
A second problem is that under our original measure $\mathbb{P}$, suitable choices of the distribution $B$ are generally defective. We bypass this by identifying an alternative measure under which the ruin occurs almost surely and $B$ is no longer defective.


Now, our main goal is to identify a martingale $\{ M_t^{(r)} \}_{t \geq 0}$ and the corresponding alternative measure under which ruin almost surely occurs. For this, we follow the ansatz of \cite{Pojer.2022}, that is, $M_t^{(r)} = \exp(-rX_t - \alpha \lambda_t -\theta t).$ This process is a local martingale if the function $h_r(x,\lambda,t) = \exp(-rx-\alpha \lambda - \theta t)$ is in the domain of the extended generator and satisfies
\[\mathcal{A}h_r(x,\lambda,t)=0.\]
We start from the latter requirement, which by \eqref{generator} is equivalent to
\begin{multline*}
    \mathcal{A}h_r(x,\lambda,t) = -cr h_r(x,\lambda,t)- \alpha \beta (a-\lambda) h_r(x,\lambda,t) - \theta h_r(x,\lambda,t) \\ +\lambda h_r(x,\lambda,t) M_U(r) M_Y(-\alpha) - \lambda h_r(x,\lambda,t) = 0,
\end{multline*}
for all choices of $x,\lambda,t$. Since $h_r$ is positive, we can divide by $h_r$ and get the following two equations
\begin{align}
    -cr - \alpha \beta a - \theta &=0, \label{firsteq}\\
    \alpha \beta+ M_U(r) M_Y(-\alpha) -1 &=0. \label{maineq}
\end{align}
For fixed $r$, we get two equations for two missing variables $\theta(r)$ and $\alpha(r)$.
We focus on equation \eqref{maineq} defining $\alpha(r)$.
\begin{lem}\label{lem:distinct_alpha}
    For $r \leq 0$ and some $r>0$, there exist two distinct solutions to the equation \eqref{maineq}.
    \end{lem}
    \begin{proof}
        First, we consider the case $r < 0$. The function
        \[f_r(\alpha) :=\alpha \beta+ M_U(r) M_Y(-\alpha) -1\]
        is convex and satisfies $f_r(0)= M_U(r)-1 <0$. Furthermore, $\lim_{\alpha \to \infty} f_r(\alpha) = \lim_{\alpha \to -\infty} f_r(\alpha)=\infty$. By continuity, there exists at least one root in $(-\infty, 0)$ and one root in $(0, \infty)$. By convexity, the corresponding roots are unique.
        For $r=0$, we have $f_0(0) = 0$ and $\frac{\partial}{\partial \alpha}f_0(0)= \beta - \mathbb{E}[Y]>0$. By this, there exists some $\varepsilon>0$ such that $f_0(-\varepsilon)<0$ and, by the same argumentation as before, we have that there exists a unique negative root of $f_0$. For $r>0$, we see that the function $f_r(\alpha)$ is also continuous in $r$ (as long as it is well defined). Consequently, there exists some $\delta >0$ such that for all $r< \delta$ we have $f_r(-\varepsilon)<0$. Again, by continuity and convexity in $\alpha$, we get the existence of two solutions to $f_r(\alpha)=0$.
    \end{proof}
\begin{definition}\label{defalpha}
    For fixed $r$, we define $\alpha(r)$ as the maximal solution to equation \eqref{maineq}.
    This is well defined for all $r \leq r_{\rm max}$, where $r_{\rm max}$ satisfies \[\min_{\alpha} \left(\alpha \beta+ M_U(r_{\rm max}) M_Y(-\alpha) -1\right) =0.\] Further, we define the function
    \begin{equation}\label{theta}\theta(r) := -cr-\alpha(r) \beta a .\end{equation}
\end{definition}
\begin{lem}\label{lem:alpha_differentiable}
The mapping $r \to \alpha(r)$ is concave and differentiable on $(-\infty, r_{\rm max})$. Furthermore, it satisfies $\alpha(0)=0$ and $r\alpha(r) <0$ for all $r \neq 0$ such that $\alpha(r)$ is well defined.\end{lem}
\begin{proof}
    By the proof of Lemma \ref{lem:distinct_alpha}, we have that $f_0$ has a negative root and satisfies $f_0(0)=0$. Thus, $\alpha(0)=0$.
    For $r>0$, we have $f_r(0)= M_U(r)-1>0$, which gives that all roots must be negative and for $r<0$ $f_r(0)= M_U(r)-1<0$.
    Therefore there exists a positive root, and for all $r \neq 0$ for which $\alpha(r)$ is well defined, we have $r\alpha(r) <0$.

    To show the concavity of $\alpha(r)$, we first show that the function $f(r,\alpha):= f_r(\alpha)$ is convex as a function of $(r, \alpha)$ from $(-\infty, s_U) \times (-s_Y, \infty)$ to $\mathbb{R}$, where $s_Y$ and $s_U$ are defined in Assumptions \ref{assumption1} and \ref{assumption2}, respectively
    (it is even convex and proper as function from $\mathbb{R}^2 \to \mathbb{R}\cup \left\lbrace \infty \right\rbrace$ if we set $f(r,\alpha) = \infty$ for all $(r,\alpha)$ outside $(-\infty, s_U) \times (-s_Y, \infty)$). To show this, we consider the Hessian $$\mathcal{H}=\begin{bmatrix}
M_U''(r)M_Y(-\alpha) & -M_U'(r)M_Y'(-\alpha) \\
-M_U'(r)M_Y'(-\alpha) & M_Y''(-\alpha)M_U
\end{bmatrix},$$ which has only non-negative eigenvalues by the log-convexity of the moment generating functions.

If we now take some $r\geq s$ and $\lambda \in [0,1]$ such that $\alpha(r)$ is well defined, we find that also $\alpha(\lambda r + (1-\lambda)s)$ is well defined and satisfies $f(\lambda r + (1-\lambda) s, \alpha(\lambda r+ (1-\lambda) s)) =0.$ Moreover, since it is the maximal root, for all $\alpha > \alpha(\lambda r + (1-\lambda)s)$ we have $f(\lambda r + (1-\lambda )s, \alpha) >0$. By the convexity of the function $f$ we get \[
    f(\lambda r + (1-\lambda)s, \lambda \alpha(r) + (1-\lambda) \alpha(s)) \leq \lambda f(r,\alpha(r)) + (1-\lambda) f(s, \alpha(s)) = 0.
\]
Consequently, $\lambda \alpha(r) + (1-\lambda) \alpha(s) \leq \alpha(\lambda r+ (1-\lambda)s).$

We still have to show that the function $r \to \alpha(r)$ is differentiable. By concavity, it is differentiable almost everywhere. To be specific, everywhere except some countable set and at every other point, the one-sided limits exist, but do not coincide. Let $r$ be such that the derivative of $\alpha(r)$ exists. Then, we get that \[ \beta \alpha'(r) + M_U'(r)M_Y(-\alpha(r)) - \alpha'(r) M_U(r)M_Y'(-\alpha(r)) = 0,\] which is, if $\beta -  M_U(r)M_Y'(-\alpha(r)) \neq 0$, equivalent to \[ \alpha'(r) = -\frac{M_U'(r)M_Y(-\alpha(r))}{\beta -  M_U(r)M_Y'(-\alpha(r)) }.\] This is the case for all $r < r_{\rm max}$. In the case $r=r_{\rm max}$, the root $\alpha(r_{\rm max})$ also minimizes $f_{r_{\rm max}}(\alpha)$. Since this function is convex and differentiable in $\alpha$, we have $0= \frac{\partial}{\partial \alpha} f_{r_{\rm max}}(\alpha(r_{\rm max})) = \beta - M_U(r_{\rm max}) M_Y'(-\alpha(r_{\rm max}))$. For all $r< r_{\rm max}$, the root $\alpha(r)$ is not the minimizer; therefore, $\alpha'(r)$ is well defined and continuous. By the continuity of the derivative, we find that $\alpha(r)$ is differentiable at every point $r < r_{\rm max}$.
\end{proof}
We recall that the function $\theta$ is defined in \eqref{theta}.
\begin{lem}\label{lem:theta_convex}
    The function $\theta(r)$ is convex, differentiable and satisfies $\theta(0)=0$ and $\theta'(0)<0$.
    \end{lem}
    \begin{proof}
        Since $\theta(r)=-cr- \alpha(r) a \beta$ is the sum of two differentiable and convex functions, it is differentiable and convex as well and $\theta(0) =0$. The derivative at the point $r=0$ is
        \[
            \theta'(0)=-c - \alpha'(r) a \beta = -c + \frac{\beta a \mathbb{E}[U]}{\beta - \mathbb{E}[Y]},
        \]
        which is negative by the net profit condition. This completes the proof.
    \end{proof}
For further analysis we will also need the following important assumption.
\begin{assumption}\label{assumption3}
    From now on, we assume that there exists a positive solution $R$ of
    \begin{equation}\label{adjustment2} \theta(r)=0.
    \end{equation}
    for the function $\theta$ defined in \eqref{theta}.
     Further, we assume that there exists an $\varepsilon>0$ such that $M_U(R+\varepsilon)$, $M_Y(-\alpha(R+\varepsilon))$ are finite.
\end{assumption}

\begin{thm}\label{localmartingale1}
    The process
    \begin{equation}\label{martingale}
    M^{(r)}_t:= \exp(ru+\alpha(r)\lambda) \exp(-rX_t-\alpha(r)\lambda_t - \theta(r) t) \end{equation}
    is a non-negative local martingale for all $0 \leq r \leq R+ \varepsilon$.\end{thm}
    \begin{proof}
        Fix $r \leq R+\varepsilon$ arbitrary and define the function \[h_r(x,\tilde \lambda, t) :=\exp(-r(x-u)-\alpha(r)(\tilde\lambda-\lambda) - \theta(r) t).\]
        We show that this function is in the domain of the extended generator of our PDMP and satisfies $\mathcal{A}h_r(x,\tilde\lambda,t)=0$, which gives us that this is a local martingale. The function $h_r$ is absolutely continuous.
        Hence, by Theorem (26.14) and Remark (26.16) of \cite{Davis.1993}, we only have to show that for all $n \in \mathbb N$ \[ \mathbb{E}_{(u,\lambda)} \left[ \sum_{i=1}^n \left\vert h_r(X_{T_i}, \lambda_{T_i}, T_i) - h_r(X_{T_i-}, \lambda_{T_i-}, T_i)\right\vert \right] < +\infty.\] This is obviously satisfied for $r \leq 0$. For the case $r>0$, we observe that $\alpha(r) <0$ and the compensator of the jump process $\{N_t\}_{t\geq 0}$ is given by \begin{align*}
            \Lambda_t:=\int_0^t \lambda_s \, \mathrm{d}s &= at + \frac{1}{\beta} (\lambda_0-a)(1-e^{- \beta t}) + \frac{1}{\beta} \sum_{k=1}^{N_t} Y_k (1-e^{-\beta (t-T_k)}) \\&= at + \frac{1}{\beta} \sum_{k=1}^{N_t} Y_k + \frac{1}{\beta}\lambda_0 - \frac{1}{\beta}\lambda_t.
        \end{align*}
Consequently, we have for all $i \leq n$,
\begin{multline*}
-r (X_{T_i}-u) - \alpha(r) (\lambda_{T_i}-\lambda) - \theta (r) T_i \\ = -rc T_i +r \sum_{j=1}^{i} U_j - \alpha(r)\left( \sum_{j=1}^{i} Y_j + \beta a  T_i - \beta \Lambda_{T_i}\right) + crT_i + a \alpha(r) \beta T_i \\\leq r \sum_{j=1}^{i} U_j - \alpha(r) \sum_{j=1}^{i} Y_j.
\end{multline*}

Further, observe that \[ h_r(X_{T_i}, \lambda_{T_i}, T_i) = h_r(X_{T_i-}, \lambda_{T_i-}, T_i) \exp\left(rU_i - \alpha(r) Y_i \right) > h_r(X_{T_i-}, \lambda_{T_i-}, T_i).\]
Using this, we get that
\begin{align*}
            &\mathbb{E}_{(u,\lambda)}\left[\sum_{i=1}^{n} \left\vert h_r(X_{T_i}, \lambda_{T_i}, T_i) -h_r(X_{T_i-}, \lambda_{T_i-}, T_i)\right\vert \right]
            \\& \leq 2\,\mathbb{E}_{(u,\lambda)}\left[\sum_{i=1}^{n}  h_r(X_{T_i}, \lambda_{T_i}, T_i) \right] \leq  2 \sum_{i=1}^n\mathbb{E}_{(u,\lambda)}\left[ \exp \left(r \sum_{j=1}^{i} U_j - \alpha(r) \sum_{j=1}^{i} Y_j\right)\right] \\& =2\sum_{i=1}^n M_U(r)^i M_Y(-\alpha(r))^i = 2M_U(r)M_Y(-\alpha(r)) \frac{M_U(r)^{n} M_Y(-\alpha(r))^{n}-1}{M_U(r) M_Y(-\alpha(r)) - 1} < +\infty.
        \end{align*}
        By this, the function $h_r$ is in the domain of the extended generator and by the construction of $\alpha(r)$ and $\theta(r)$ it satisfies $\mathcal{A}h_r(x,\lambda,t)=0$.
        This completes the proof.
    \end{proof}
\begin{thm}\label{localmartingale2}
Let $r<R+\varepsilon$ for some $\varepsilon >0$. Then, $M^{(r)}=\{ M^{(r)}_t\}_{t \geq 0}$ defined in \eqref{martingale} is a true martingale with expectation $1$.
\end{thm}
\begin{proof}
Fix $r< R+\varepsilon$ and let $\{ \varrho_n \}_{n \in \mathbb{N}}$ be a localizing sequence of stopping times for the local martingale $\{ M^{(r)}_t \}_{t \geq 0}$. Then, by Lemma 2.2.2 of \cite{Fleming.1991}, we have that $\{ M^{(r)}_t \}_{t \geq 0}$ is a martingale if for any fixed $t$, the family $\mathcal{X} = \{ M^{(r)}_{t \wedge \varrho_n} \}_{n\in \mathbb{N}}$ is uniformly integrable. By de La Vallée Poussin’s Theorem, a family of random variables $\{ Y_n \}_{n \in A}$ is uniformly integrable if there exists a monotone increasing convex function $G(t)$, satisfying $\lim_{t\to \infty} \frac{G(t)}{t}=\infty$ and $\sup_{n \in A} \mathbb{E}\left[G(\left\vert Y_n\right\vert )\right]< +\infty$.

Since $r< R+\varepsilon$, there exists a $\delta >0$ such that $r(1+\delta) < R+\varepsilon$ and $\{M_t^{(r(1+\delta))}\}_{t \geq 0}$ is well-defined. Since every non-negative local martingale with integrable initial value is a supermartingale, we have that for all $t\geq 0$ that $M^{(r(1+\delta))}_t$ is integrable with expectation less or equal $M^{(r(1+\delta))}_0 =1$. \\
Let now $t$ and $n$ be arbitrary but fixed. Then, we have that \begin{multline*}
\left(M^{(r)}_{t \wedge\varrho_n}\right)^{1+\delta}=\exp\left( -r(1+\delta)(X_{t \wedge\varrho_n}-u) - \alpha(r) (1+\delta) (\lambda_{t \wedge\varrho_n}-\lambda) \right) \\\times \exp \left( cr(1+\delta) t\wedge\varrho_n + a \beta \alpha(r) (1+\delta) t \wedge\varrho_n\right) \\=
M^{(r(1+\delta))}_{t \wedge\varrho_n} \,\exp\left((\alpha(r(1+\delta)) - \alpha(r)(1+\delta)) \lambda_{t \wedge\varrho_n} \right)  \\ \times \exp\left(( \alpha(r)(1+\delta)-\alpha(r(1+\delta))) (\lambda + a \beta t \wedge \varrho)\right).
\end{multline*}
If we can show that $\alpha(r(1+\delta)) - \alpha(r)(1+\delta)\leq 0$, then we have by the positivity of $\lambda_{t \wedge \varrho_n}$ that \[ \exp\left((\alpha(r(1+\delta)) - \alpha(r)(1+\delta)) \lambda_{t \wedge\varrho_n} \right) \leq 1,\]
and, since $t\wedge\varrho_n \leq t$, we would have that
\begin{align*}
    &\exp\left( \alpha(r)(1+\delta)-\alpha(r(1+\delta))) (\lambda_0 + a \beta t \wedge \varrho)\right)\\&\quad \leq \exp\left(( \alpha(r)(1+\delta)-\alpha(r(1+\delta))) \right)  \exp\left((\lambda_0 + a \beta t )\right),
\end{align*}
which is deterministic and finite. \\
To show this, we will use the fact that $\alpha(r)$ is concave, differentiable and satisfies $\alpha(0)=0$; see Lemma \ref{lem:alpha_differentiable}. By this, we get that \begin{multline*}
\alpha(r(1+\delta)) - \alpha(r)(1+\delta) = \alpha(r(1+\delta)) - \alpha(r)- \alpha(r)\delta \leq \alpha'(r) r\delta - \alpha(r) \delta \\ = \alpha'(r) r\delta  + \delta ( \alpha(0) - \alpha(r)) \leq \delta \alpha'(r)r \delta + \delta \alpha'(r) (-r) =0.
\end{multline*}
This gives that
\begin{align*}
&\mathbb{E}_{(u,\lambda)} \left[ \left(M^{(r)}_{t \wedge\varrho_n}\right)^{1+\delta} \right] \leq \mathbb{E}_{(u,\lambda)}
\left[ \left(M^{(r(1+\delta))}_{t \wedge\varrho_n}\right) \right]
\\&\qquad\times
 \exp\left( \alpha(r)(1+\delta)-\alpha(r(1+\delta))) (\lambda_0 + a \beta t )\right) \\&\quad \leq \exp \left( \alpha(r)(1+\delta)-\alpha(r(1+\delta))) (\lambda_0 + a \beta t )\right).
\end{align*}
This bound is independent of $n$ and finite for fixed $t$. Hence, taking the supremum gives us that \[ \sup_{n \in \mathbb{N}} \mathbb{E}_{(u,\lambda)} \left[ \left(M^{(r)}_{t \wedge\varrho_n}\right)^{1+\delta} \right] < +\infty.\] Therefore, $ \mathcal{X}$ is uniformly integrable and the process $\{ M^{(r)}_t \}_{t \geq 0}$ is a true martingale with expectation $M^{(r)}_0 =1$.
\end{proof}
\begin{definition}\label{newmeasuredef}
    Let $R+\varepsilon > r \geq 0$ for some $\varepsilon >0$.
    Then, we define the measure $\mathbb{Q}^{(r)}$ by \[ \mathbb{Q}^{(r)} \left[A\right]= \mathbb{E}_{(u,\lambda)} \left[I_{A} M_t^{(r)}\right], \, \forall A \in \mathcal{F}_t. \]
\end{definition}
    \begin{lem}\label{newgenerator}
Under the new measure $\mathbb{Q}^{(r)}$, the multivariate process $\{ (X_t, \lambda_t, t)\}_{t\geq 0}$ is again a PDMP with generator \begin{multline*}
    \mathcal{A}^{(r)} f(x, \lambda, t) = c \frac{\partial}{\partial x}f(x,\lambda, t) + \beta (a-\lambda) \frac{\partial}{\partial \lambda}f(x,\lambda,t) + \frac{\partial}{\partial t}f(x,\lambda,t)\\ + \lambda \int_0^\infty \int_0^\infty e^{ru} e^{-\alpha(r) y} \left(f(x-u,\lambda+y,t)-f(x,\lambda,t)\right)) \, F_U(\mathrm{d}u) \, F_Y(\mathrm{d}y)  .
\end{multline*}
    \end{lem}
        \begin{proof}
            This follows directly from Example 5.2 of \cite{Palmowski.2002}, where exactly this kind of exponential change of measures for PDMPs is studied.
        \end{proof}
 \begin{lem}\label{driftpostive}
Under the new measure $\mathbb{Q}^{(R)}$, ruin occurs almost surely.
    \end{lem}\label{newmeasurerecc}
        \begin{proof}
            Using the same ideas as in Lemma \ref{lem:mean_surplus} but with the alternative generator $\mathcal{A}^{(R)}$, it is easy to see that \[ \lim_{t \to \infty} \frac{\mathbb{E}^{(R)}_{(u,\lambda)}\left[X_t\right]}{t}= - \theta'(R).\]
            By the convexity of $\theta$ proved in Lemma \ref{lem:theta_convex} and the fact that there is some $r<R$ with $\theta(r)<0$, we have that $-\theta'(R) <0$. Consequently, ruin occurs almost surely under the new measure $\mathbb{Q}^{(R)}$.
        \end{proof}
    We will that under the new measure $\mathbb{Q}^{(r)}$, $\{ \lambda_t\}_{t \geq 0}$ is no longer the intensity of a Markovian marked Hawkes process $\{N_t\}_{t\geq 0}$.
     In fact its jumps have now intensity $\{ \lambda_t M_U(r)M_Y(-\alpha(r)) \}_{t \geq 0}$ that still preserves its recurrent behaviour.
    \begin{lem}\label{Qrparamneters}
        The process $\{ \lambda_t\}_{t \geq 0}$ is Harris recurrent under $\mathbb{Q}^{(r)}$.    \end{lem}
        \begin{proof}
            At first, we show that, under the measure $\mathbb{Q}^{(r)}$, the process \linebreak $\{ M_U(r)M_Y(-\alpha(r)) \lambda_t \}_{t\geq 0}$ is the intensity of a Markovian marked Hawkes process $\{N_t\}_{t\geq 0}$.
            Indeed, from the form of the generator $\mathcal{A}^{(r)}$ given in Lemma \ref{newgenerator} we can conclude that the univariate process $\{ \lambda_t \}_{t \geq 0}$ is a Markov process with
            generator \[ \mathcal{A}^{(r),\lambda} f(\lambda) = \beta (a-\lambda) f'(\lambda) + \lambda M_u(r) M_Y(-\alpha(r)) \int_0^\infty (f(\lambda+y) - f(\lambda)) \, \tilde F_Y(\mathrm{d}y),\] where the distribution $\tilde F_Y$ is given by $\tilde F_Y(\mathrm{d}y) = \frac{e^{-\alpha(r)y}}{M_Y(-\alpha(r))}F_Y(\mathrm{d}y)$. Hence, under $\mathbb Q^{(r)}$, the process $\{ \lambda_t \}_{t \geq 0}$ has the form \[ \lambda_t = e^{-\beta t} (\lambda - a) + a + \sum_{i=1}^{N^{(r)}_t} \tilde Y_i e^{- \beta (t- T^{(r)}_i)}, \] where $\{ N^{(r)}_t\}_{t \geq 0}$ has the intensity process $\{ M_U(r) M_Y(-\alpha(r)) \lambda_t \}_{t \geq 0}$.

            Observe that the PDMP $\{ M_U(r)M_Y(-\alpha(r)) \lambda_t \}_{t\geq 0}$ can represented as follows
            \begin{multline*}
                M_U(r)M_Y(-\alpha(r)) \lambda_t = e^{-\beta t} ( M_U(r)M_Y(-\alpha(r))\lambda -  M_U(r)M_Y(-\alpha(r))a) \\+  M_U(r)M_Y(-\alpha(r))a + \sum_{i=1}^{N^{(r)}_t}  M_U(r)M_Y(-\alpha(r)) \tilde Y_i e^{- \beta (t- T^{(r)}_t},
            \end{multline*}  that is,
            the process $\{ N^{(r)}_t \}_{t \geq 0}$ is a Markovian Hawkes process.
            The parameters are given by decay parameter $\beta$, baseline intensity $a M_U(r)M_Y(-\alpha(r))$ and shock distribution $\tilde F_Y(z/(M_U(r)M_Y(-\alpha(r))))$. \\
            By Theorem \ref{thm:posrecc} and \eqref{dodzal}, if we can now show that
            \begin{equation}\label{inrecc}\beta > M_U(r)M_Y(-\alpha(r))\mathbb{E}^{\mathbb{Q}^{(r)}}\left[ Y\right] = M_U(r)M_Y(-\alpha(r))\mathbb{E}\left[ \tilde Y\right],\end{equation}
            then $\{ M_U(r)M_Y(-\alpha(r)) \lambda_t \}_{t\geq 0}$ returns to every point in $(a M_U(r)M_Y(-\alpha(r)), \infty)$ infinitely often. This implies that $\{ \lambda_t\}_{t \geq 0}$ visits every point in $(a,\infty)$ infinitely often.

            To prove \eqref{inrecc}, observe first that the expectation of $Y$ under our new measure is $\frac{M_Y'(-\alpha(r))}{M_Y(-\alpha(r))}$ and that the mapping $r \to \alpha(r)$ is monotone decreasing. Hence, by the proof of Lemma \ref{lem:alpha_differentiable} we have that $ - \frac{M_U'(r)M_Y(-\alpha(r))}{\beta - M_U(r)M_Y'(-\alpha(r))}\leq 0$, which implies that \[\beta - M_U(r)M_Y(-\alpha(r)) \mathbb E [\tilde Y] = \beta -M_U(r)M_Y'(-\alpha(r)) >0 .\] 
            By this, we have that our intensity process returns almost surely to every point in  $(aM_U(r)M_Y(-\alpha(r)), \infty)$.
        \end{proof}
\subsection{Cram\'er-Lundberg asymptotics and renewal arguments}\label{subsec:CL}
Now, fix an initial value $\lambda$ and let $S^\lambda_1, \ldots$ denote the recurrence times of the intensity process to the level $\lambda$, i.e. $\lambda_{S^\lambda_i} = \lambda$ for all $i$. Further, we define the renewal times $\{ \phi_i \}_{i \geq 1}$ by $\phi_1 = \min \left\lbrace \, S^\lambda_i: \, X_{S^\lambda_i}<u \, \right\rbrace$ and $\phi_j = \min \left\lbrace\, S^\lambda_i: \, X_{S^\lambda_i}< X_{\phi_{j-1}}\, \right\rbrace $ for $j>1$.

These times are a mixture of the recurrence times of the intensity process and ladder times of the surplus process, i.e. ladder times of the random process $\{ X_{S^\lambda_i}\}_{i \geq 1}$. As we can see, these renewal times might be infinite under our original measure since $X_t \to +\infty$ $\mathbb{P}$-a.s.
But under our alternative measure $\mathbb{Q}^{(R)}$, the surplus process $\{X_t\}_{t\geq 0}$ tends to $-\infty$, and the intensity returns infinitely often to $\lambda$. Hence, these times are finite almost surely. Define now
\[\tilde{B}(x):= \mathbb{P}_{(u,\lambda)} \left[ \phi_1 < +\infty, u-X_{\phi_1} \leq x\right]\]
and
\[p(u,x):= \mathbb{P}_{(u,\lambda)}\left[\tau \leq \phi_1, \, \left\vert \phi_1 < +\infty, X_{\phi_1} =u-x \right.\right].\]
Then we have, by conditioning on the distribution of the surplus at time $\phi_1$, that \[ \psi(u,\lambda) = \int_0^u \psi(u-x,\lambda) (1-p(u,x)) \, \tilde{B}(\mathrm{d}x) + \mathbb{P}_{(u,\lambda)}\left[\tau\leq \phi_1, \tau < +\infty \right].\]
As already mentioned, the distribution $\tilde{B}$ is defective. To work with a proper distribution,
we multiply the equation by $e^{Ru}$ to obtain
\begin{equation}\label{renewalruin}
\psi(u,\lambda) e^{Ru}= \int_0^u \psi(u-x,\lambda)e^{R(u-x)} (1-p(u,x)) \, \tilde{B}(\mathrm{d}x) + e^{Ru}\mathbb{P}_{(u,\lambda)}\left[\tau\leq \phi_1, \tau < +\infty, \right]\end{equation}
where
\[B(\mathrm{d}x) := e^{Rx} \tilde{B}(\mathrm{d}x).\]
\begin{lem}\label{lem:mass_1}
    The distribution $B$ is a proper probability distribution.
    \end{lem}
    \begin{proof}
    By the definition of $\tilde{B}$, we have that
    \[
        \int_\mathbb{R}e^{Rx} \, \tilde{B}(\mathrm{d}x) = \mathbb{E}_{(u,\lambda)}\left[e^{R(u-X_{\phi_1})}I_{\left\lbrace \phi_1 < +\infty\right\rbrace}\right].
    \]
        Our martingale $\{ M^{(R)}_t\}_{t \geq 0}$ at time $\phi_1$ has the form \[ M^{(R)}_{\phi_1} = \exp(-R(X_{\phi_1} - u) - \alpha(R) (\lambda_{\phi_1}-\lambda)) = \exp(R(u-X_{\phi_1})).\] Therefore,
        \[  \int_\mathbb{R} \, B(\mathrm{d}x)=\int_\mathbb{R}e^{Rx} \, \tilde{B}(\mathrm{d}x) = \mathbb{Q}^{(R)}\left[ \phi_1 < +\infty \right]=1. \]
    \end{proof}
Observe that equation \eqref{renewalruin} is of the form of renewal equation
\eqref{renewalequation}, that is,
\begin{equation}\label{renewalequation2}
    Z(u)= \int_0^u Z(u-x) (1-p(u,x)) B(\mathrm{d}x) + z(u)
\end{equation}
for
\[Z(u):=\psi(u,\lambda) e^{Ru}\quad\text{and}\quad z(u):= e^{Ru}\mathbb{P}_{(u,\lambda)}\left[\tau\leq \phi_1, \tau < +\infty \right].\]
To show convergence of $Z(u)$, hence Cram\'er-Lundberg asymptotics, we have to verify that
$z(u)$ and $\int_0^u p(u,x) B(\mathrm{d}x)$ are directly Riemann integrable.
For the direct Riemann integrability of above-mentioned functions, we need to introduce an additional assumption. \begin{assumption}\label{ass:exponentially_rec_time}
    We assume there exists an $\varepsilon>0$ such that \[ \mathbb{E}_{(u,\lambda)}\left[ e^{-(1+\varepsilon)R (X_{\phi_1}-u)}I_{\left\lbrace \phi_1 < +\infty\right\rbrace}\right]< +\infty.\]
\end{assumption}
\begin{remark}\rm The random time $\phi_1$ depends on the behaviour of the bivariate process $\{ (X_t, \lambda_t) \}_{t \geq 0}$. Therefore, this assumption may be hard to check. An alternative to this is the condition \[ \mathbb{E}_{(u,\lambda)}\left[ e^{-(1+\varepsilon)R (X_{S^\lambda_1}-u)}I_{\left\lbrace S^\lambda_1 < +\infty\right\rbrace}\right]< +\infty,\] which is equivalent to Assumption \ref{ass:exponentially_rec_time} by Lemma 10 of \cite{Pojer.2022}.
Changing measure, we see that this assumption is equivalent to
\begin{equation}\label{suff0} \mathbb E^{\mathbb Q^{(1+\varepsilon)R}}\left[e^{\theta((1+\varepsilon)R) \phi_1}\right] < +\infty, \end{equation}
which shows the main influence of this assumption. By the structure of our renewal times, we cannot observe ruin exactly when it happens. Assumption \ref{ass:exponentially_rec_time} ensures that these renewal times happen often enough, such that there is one of these close enough to the time of ruin such that we do not miss the event that the surplus process is negative.
\end{remark}
Now we will show that the functions $z(u)$ and $\int_0^u p(u,x) B(\mathrm{d}x)$
are directly Riemann integrable. To do so, we will use Proposition V.4.1 on p. 154 of \cite{Asmussen.1995}, which gives us that it is sufficient to prove that both considered functions are continuous and there exists bounded directly Riemann integrable upper bounds for these functions.
\begin{lem}\label{lem:dri_z}
    Under Assumptions \ref{assumption1}-\ref{ass:exponentially_rec_time}, we have that \[z(u)=
    e^{Ru}\mathbb{P}_{(u,\lambda)}\left[\tau\leq \phi_1, \tau < +\infty \right]\] is directly Riemann integrable.
    \end{lem}
    \begin{proof}
    We start the proof from showing that the function $z(u)$ is continuous. Indeed, from \eqref{renewalequation2} it follows that it suffices to show continuity of the ruin probability $\psi(u,\lambda)$ as a function of $u>0$.
For $h>0$, by Markov property of $\{(X_t,\lambda_t,t)\}_{t\geq 0}$, we have,
\begin{align}\label{basicid}
\psi(u,\lambda)&=\mathbb{P}(N_h=0) \psi(u+ch, \lambda)
\\&\quad+\sum_{k=1}^\infty \mathbb{P}(N_h=k) \mathbb{E}\psi(u+c(T_1+\ldots+T_k)-U_1-\ldots -U_k).\nonumber
\end{align}
Observe that $\lim_{h\to 0} \mathbb{P}(N_h=k)=0$ for $k\in \mathbb{N}$ and $\lim_{h\to 0}\mathbb{P}(N_h=0)=1$; see e.g.
\cite{Hawkes}.
Hence, by Lemma \ref{upperboundgen} and the dominated convergence theorem, we can conclude that
$\lim_{h\to 0}\sum_{k=1}^\infty \mathbb{P}(N_h=k) \mathbb{E}\psi(u+c(T_1+\ldots+T_k)-U_1-\ldots -U_k)=0$
and that $\psi(u,\lambda$ is right-continuous.
Plugging on the left-hand side of \eqref{basicid}, $u-ch$ instead of $u$ into the argument of $\psi$ gives the left-continuity of this function.

        Now, let $\varepsilon>0$ such that $\mathbb{E}_{(u,\lambda)}\left[ e^{-(1+\varepsilon)R (X_{\phi_1}-u)}I_{\left\lbrace \phi_1 < +\infty\right\rbrace}\right]< +\infty$. Let $r=(1+\varepsilon)R$. Observe that \begin{multline*}
            e^{ru} \mathbb{P}_{(u,\lambda)}\left[\tau\leq \phi_1, \tau < +\infty \right] =  e^{ru} \mathbb{E}_{(u,\lambda)}\left[I_{\left\lbrace \tau < \phi_1\right\rbrace} I_{\left\lbrace \tau < +\infty \right\rbrace}\right] \\=e^{ru} \mathbb{E}^{\mathbb{Q}^{(r)}}\left[I_{\left\lbrace \tau < \phi_1\right\rbrace} \exp\left(r(X_\tau-u)+\alpha(r) (\lambda_\tau - \lambda) + \theta(r) \tau \right)\right].
        \end{multline*}
        Since $r>0$ and $X_\tau <0$, we have that $\exp(rX_\tau) <1$ and the same holds for $\exp(\alpha(r) \lambda_\tau)$. Therefore, \begin{multline*}
            e^{ru} \mathbb{P}_{(u,\lambda)}\left[\tau\leq \phi_1, \tau < +\infty \right] \leq e^{-\alpha(r) \lambda}\mathbb{E}^{\mathbb{Q}^{(r)}}\left[I_{\left\lbrace \tau < \phi_1\right\rbrace} \exp\left( \theta(r) \phi_1 \right)\right] \\ \leq e^{-\alpha(r) \lambda}\mathbb{E}^{\mathbb{Q}^{(r)}}\left[I_{\left\lbrace  \phi_1< +\infty\right\rbrace} \exp\left( \theta(r) \phi_1 \right)\right] \\= e^{-\alpha(r) \lambda} \mathbb{E}_{(u,\lambda)}\left[ e^{-(1+\varepsilon)R (X_{\phi_1}-u)}I_{\left\lbrace \phi_1 < +\infty\right\rbrace}\right]< +\infty.
        \end{multline*}
        Thus, we have that there exists a positive constant $K$ such that
        \[  e^{ru} \mathbb{P}_{(u,\lambda)}\left[\tau\leq \phi_1, \tau < +\infty \right] \leq K e^{-(r-R)u}=Ke^{-\varepsilon u}\] and the upper bound is bounded nad directly Riemann integrable. This completes the proof.
    \end{proof}

\begin{lem}\label{lem:dri_integral}
    Under Assumptions \ref{assumption1}-\ref{ass:exponentially_rec_time}, the function which maps $u$ to $\int_0^u p(u,x)e^{Rx}B(\mathrm{d}x)$ is directly Riemann integrable.
    \end{lem}
    \begin{proof}
    Observe that the function $u\rightarrow \int_0^u p(u,x)e^{Rx}B(\mathrm{d}x)$ is continuous.
To identify a bounded directly Riemann integrable upper bound, we choose arbitrary but fixed $u$. Then,
\begin{multline*}
            \int_0^u p(u,x) e^{Rx} B(\mathrm{d}x) \leq e^{Ru}\int_0^u p(u,x) B(\mathrm{d}x) = e^{Ru} \mathbb{P}_{(u,\lambda)}\left[\tau\leq \phi_1, \phi_1<+\infty\right] \\\leq e^{Ru}\mathbb{P}_{(u,\lambda)}\left[\tau\leq \phi_1, \tau < +\infty \right].
        \end{multline*}
        By Lemma \ref{lem:dri_z}, we know that the upper bound is directly Riemann integrable and bounded, which completes the proof.
      \end{proof}
We are now ready to prove our next main result.
\begin{thm}\label{cramerasymptotics}
    Under Assumptions \ref{assumption1}-\ref{ass:exponentially_rec_time}, there exists a constant $C^\lambda>0$, depending on $\lambda$, such that \[ \lim_{u \to \infty} \psi(u,\lambda)e^{Ru} = C^\lambda,\]
    where the adjustment coefficient $R>0$ solves Lundberg equation \eqref{adjustment2}.
    \end{thm}
    \begin{proof}
        By the absolute continuity of the claim events and the proof of Lemma 12 of \cite{Pojer.2022}, we have that $p(u,x)$ is continuous in $u$. By this and Lemmas \ref{lem:mass_1}-\ref{lem:dri_integral}, all assumptions of Theorem 2 of \cite{Schmidli.1997} are satisfied, which gives the statement of our theorem. 
    \end{proof}
\begin{remark}\rm
By Corollary \ref{corbounds}, the adjustment coefficient $R>0$ defined via \eqref{adjustment2} equals to the adjustment coefficient
of 'shifted' Cram\'er-Lunberg risk process defined in \eqref{adjustmentcoeff}.
\end{remark}
\begin{remark}
Theorem \ref{cramerasymptotics} gives a stronger statement than Theorem 4.1 of \cite{Stabile} who derived
only the logarithmic asymptotic showing that $\lim_{u\to+\infty}\frac{1}{u}\ln \psi(u)=R$.
\end{remark}

\section{ Markovian Hawkes arrival process with exponentially distributed shocks and exponential claims}\label{sec:example}
Here, we introduce an example where all Assumptions \ref{assumption1}-\ref{ass:exponentially_rec_time} are satisfied. For this, we consider a Markovian Hawkes process with the intensity process \eqref{lambdamarkov} and with exponentially distributed shocks \[Y_i \sim {\rm Exp}(\gamma).\]
To ensure that the integrability condition $\frac{\mathbb{E}\left[Y\right]}{\beta} <1$ given in \eqref{dodzal}
is satisfied, we assume that
\[\beta \gamma >1.\]
\subsubsection*{The stationary distribution}
We start from the following fact which is of own interest.
\begin{thm}
The stationary measure $\nu$ of the intensity process $\{\lambda_t\}_{t\geq 0}$ exists
 and it is shifted Gamma law, that is,
\begin{equation}\label{gammalaw} \nu \sim a+ {\rm Gamma} (a/\beta, (\beta \gamma -1)/\beta).\end{equation}
\end{thm}
\begin{proof}
By Theorem 34.19 on p. 118 of \cite{Davis.1984} (see also Prop. 34.7, p. 113 and Prop. 34.11, p. 115 of  \cite{Davis.1984}) and the stationary distribution $\nu$ of the PDMP with density $p$ satisfies
\[ 0=\int_a^\infty \mathcal{A}f(x) \nu(\mathrm{d}x) = \int_a^\infty f(x) \mathcal{A}^*p(x) \, \mathrm{d}x,\] for all $f\in D(\mathcal{A})$ in the domain of the generator $\mathcal{A}$, where $\mathcal{A}^*$ is an adjoint operator to $\mathcal{A}$.
If we can find the unique solution to the equation
\begin{equation}\label{eqstationaryexp}
\mathcal{A}^*g(\lambda)=0,
\end{equation}
then, by the uniqueness of the stationary distribution, this solution must be a density of the stationary distribution.

We recall that
\begin{multline*}
    \mathcal{A}f(\lambda) = \beta (a-\lambda) f'(\lambda) + \lambda \int_0^\infty \gamma e^{-\gamma y} \left( f(\lambda +y) - f(\lambda) \right) \, \mathrm{d}y \\ = \beta (a-\lambda) f'(\lambda) + \lambda \int_\lambda^\infty \gamma e^{-\gamma (y-\lambda)}  f(y)   \, \mathrm{d}y -\lambda f(\lambda)
\end{multline*}
and the adjoint operator $\mathcal{A}^*$ satisfies
\[\int_a^\infty (\mathcal{A}f(\lambda)) g(\lambda) \, \mathrm{d}\lambda = \int_a^\infty (f(\lambda)) \mathcal{A}^*g(\lambda) \, \mathrm{d}\lambda,\] for all functions $f$ and $g$ from the domain of $\mathcal{A}$. Therefore, for $b>a$,
\begin{multline*}
    \int_a^b (\mathcal{A}f(\lambda)) g(\lambda) \, \mathrm{d}\lambda = \int_a^b \beta (a-\lambda) f'(\lambda) g(\lambda) \, \mathrm{d}\lambda \\ + \int_a^b \lambda g(\lambda) \int_\lambda^\infty \gamma e^{-\gamma (y-\lambda)} f(y) \, \mathrm{d}y \, \mathrm{d}\lambda- \int_a^b \lambda f(\lambda) g(\lambda) \, \mathrm{d}y\lambda.
\end{multline*}
Furthermore,. if we use integration by parts in the first integral, we get \[ \int_a^b \beta (a-\lambda) f'(\lambda) g(\lambda) = f(b) \beta (a-b) g(b) - \int_a^b f(\lambda) \left(\beta (a-\lambda) g'(\lambda) -\beta g(\lambda) \right) \, \mathrm{d}\lambda. \]
In the second term, we interchange integrals and obtain \begin{multline*}
    \int_a^b \lambda g(\lambda) \int_\lambda^\infty \gamma e^{-\gamma (y-\lambda)} f(y) \, \mathrm{d}y \, \mathrm{d}\lambda = \int_a^\infty f(y) \int_a^{\min (y,b)} \lambda g(\lambda) \gamma e^{-\gamma (y-\lambda)} \, \mathrm{d}\lambda \, \mathrm{d}y \\ = \int_a^b f(\lambda) \int_a^\lambda y g(y) \gamma e^{-\gamma (\lambda-y)} \, \mathrm{d}y \, \mathrm{d}\lambda + \int_b^\infty f(\lambda) \int_a^{b} y \,g(y)\, \gamma e^{-\gamma (\lambda-y)} \, \mathrm{d}y \, \mathrm{d}\lambda.
\end{multline*}
Plugging these together, we have that \begin{multline*}
    \int_a^b (\mathcal{A}f(\lambda)) g(\lambda) \, \mathrm{d}\lambda =  f(b) \beta (a-b) g(b) + \int_b^\infty f(\lambda) \int_a^{b} y \,g(y)\, \gamma e^{-\gamma (\lambda-y)} \, \mathrm{d}y \, \mathrm{d}\lambda \\ + \int_a^b f(\lambda) \left( \beta (a-\lambda)g'(\lambda) - \beta g(\lambda) + \int_a^\lambda y \,g(y)\,\gamma e^{-\gamma (\lambda-y)} \, \mathrm{d}y-\lambda \, g(\lambda) \right) \, \mathrm{d}\lambda.
\end{multline*}
 If we let $b$ tend to infinity, the first two terms vanish and we find that the adjoint operator is given by \[\mathcal{A}^* g(\lambda) = \beta (a-\lambda)g'(\lambda) - \beta g(\lambda) + \int_a^\lambda y \,g(y)\,\gamma e^{-\gamma (\lambda-y)} \, \mathrm{d}y-\lambda \, g(\lambda). \]

 To solve equation \eqref{eqstationaryexp}, observe that \[ \frac{\partial}{\partial \lambda}\mathcal{A}^* g(\lambda) = \beta (\lambda-a) g''(\lambda)+(2\beta -\lambda) g'(\lambda) + (\beta \gamma -1) g(\lambda) - \int_a ^\lambda \gamma ^2 e^{-\gamma (\lambda-y)} \, y \, g(y) \, \mathrm{d}y \]
 and the solution of \eqref{eqstationaryexp} satisfies
 \begin{multline*}
     0 = \gamma \mathcal{A}^* g(\lambda) + \frac{\partial}{\partial \lambda} \mathcal{A}^* g(\lambda) \\ = \beta (\lambda -a ) g''(\lambda) - (\lambda +\beta(-2+a\gamma - \gamma \lambda )) g'(\lambda) + (\beta \gamma -1) g(\lambda).
 \end{multline*}
This equation has solutions of the form
\begin{multline}\label{densityformexp}
    g(\lambda) = c_1 e^{-\frac{(\beta \gamma -1)}{\beta}\lambda} (\lambda-a)^{\frac{a}{\beta}-1} \\+ c_2 e^{-\frac{(\beta \gamma -1)}{\beta}\lambda} (\lambda-a)^{\frac{a}{\beta}-1} \Gamma\left(1- \frac{a}{\beta}, (\frac{1}{\beta}-\gamma) (\lambda-a)\right),
\end{multline}
where $\Gamma$ denotes the incomplete gamma function.
To get a proper distribution from the function $g$ we have to set $c_2 =0$.
Hence $g(\lambda)$ is the density of a gamma distribution with parameters $\frac{a}{\beta}$ and $\frac{(\beta \gamma -1)}{\beta}$ and support shifted by $a$.
This completes the proof.
\end{proof}
\begin{remark}
This coincides with Remark 4.3 of \cite{Dassios.2011}, where they derived the stationary distribution using the limit of the corresponding Laplace transformations.
\end{remark}
\subsubsection*{Assumptions \ref{assumption1}- \ref{assumption3} and the form of adjustment coefficient $R$}
We now consider the surplus process \[ X_t = u+ct - \sum_{i=1}^{N_t} U_i.\]
where the claims have an exponential distribution with parameter $\mu>0$, that is, \[U_i \sim {\rm Exp}(\mu).\] In this case, the net profit condition simplifies to
\[c > \frac{a\beta \gamma}{\mu(\beta \gamma -1)},\]
and the moment generating function $M_U(r) = \frac{\mu}{\mu-r}$ is well defined for all $r<\mu$ and satisfies $\lim_{r \to \mu} M_U(r) = +\infty$.
Since $Y_i \sim {\rm  Exp}(\gamma)$, we get $M_Y(-\alpha) = \frac{\gamma}{\gamma+\alpha}$, for $\alpha >-\gamma$ and
$\lim_{\alpha \to -\gamma} M_U(\alpha) = +\infty$. Hence Assumptions \ref{assumption1}- \ref{assumption2} are satisfied.

To verify that Assumption \ref{assumption3} is also true, observe that the equations for $\theta(r)$ and $\alpha(r)$ have the form \begin{align*}
\alpha^2 \beta + \alpha(\beta\gamma -1) + \frac{\mu \gamma}{\mu-r}-\gamma &= 0, \\
\theta = -cr - \alpha a \beta.
\end{align*}
We can solve the quadratic equation for $\alpha$ and obtain the following solutions \[ \alpha_{1,2} = \frac{1-\gamma \beta}{2\beta}\pm \frac{\sqrt{(-4r\beta\gamma + (-1+\beta \gamma )^2 (\mu-r))(\mu-r)}}{2\beta (\mu-r)}.\]
As we expect from the theory already derived, there are two distinct real solutions for $\alpha$ as long $r< r_{\rm max} = \frac{(\beta \gamma -1)^2}{(\beta \gamma +1)^2}\mu$, there is one single solution for $r=r_{\rm max}$ and no real solution if $r>r_{\rm max}$.

The larger solution is  \[\alpha(r) = \frac{1-\gamma \beta}{2\beta}+ \frac{\sqrt{(-4r\beta\gamma + (-1+\beta \gamma )^2 (\mu-r))(\mu-r)}}{2\beta (\mu-r)}.\] Hence, the function $\theta$ is given by \[\theta (r) = -cr + \frac{a (\beta \gamma -1)}{2} - \frac{a\sqrt{(-4r\beta\gamma + (-1+\beta \gamma )^2 (\mu-r))(\mu-r)} }{2(\mu-r)}.\]
Solving $\theta(r) =0$ to obtain the adjustment coefficient $R$ gives us three solutions. Namely, \begin{align*}
r_1 &=0 ,\\[3pt]
r_2 &= \frac{-a+a\beta \gamma + c\mu - \sqrt{(a(1+\beta \gamma))^2 -2ac(-1+\beta \gamma)\mu +c^2\mu^2}}{2c}, \\[3pt]
r_3 &= \frac{-a+a\beta \gamma + c\mu + \sqrt{(a(1+\beta \gamma))^2 -2ac(-1+\beta \gamma)\mu +c^2\mu^2}}{2c}.
\end{align*}
This seems surprising since, by Lemma \ref{lem:theta_convex}, we know that $\theta$ is convex; hence, we would expect two roots. To resolve this puzzle, we take a closer look at the third root $r_3$ and see that \begin{multline*}
r_3 = \frac{-a+a\beta \gamma + c\mu + \sqrt{(a(1+\beta \gamma))^2 -2ac(-1+\beta \gamma)\mu +c^2\mu^2}}{2c} =\\
\frac{-a+a\beta \gamma + c\mu +\sqrt{(a(-1+\beta \gamma)-c\mu)^2 +4a^2\beta \gamma }}{2c} \\ \geq \frac{-a+a\beta \gamma + c\mu +\left \vert (a(-1+\beta \gamma)-c\mu)\right\vert}{2c}\\ \geq \frac{-a+a\beta \gamma + c\mu + (-a(-1+\beta \gamma)+c\mu)}{2c} = \mu.
\end{multline*}
In the previous parts, $\theta$ was only defined in the interval $(-\infty,r_{\rm max})$. Since $\mu>r_{\rm max}= \frac{(\beta \gamma -1)^2 \mu}{(\beta \gamma+1)^2}$, we see that the third root is not in the domain under consideration.

To ensure that the second root
\[R=r_2= \frac{-a+a\beta \gamma + c\mu - \sqrt{(a(1+\beta \gamma))^2 -2ac(-1+\beta \gamma)\mu +c^2\mu^2}}{2c}\]
(which is our adjustment coefficient) is in the domain, we must assume
the additional condition
\begin{equation}\label{additionalassumptionexp}
c< \frac{a(\beta \gamma +1)^2}{2(\beta \gamma -1)\mu}.
\end{equation}
This requirement \eqref{additionalassumptionexp} corresponds exactly to (4.5) in \cite{Karabash.2015}, which was needed to show the convergence of the logarithm of the probability of ruin in the general Hawkes case.
Finally, note that there exists an $\varepsilon>0$ such that $M_U(R+\varepsilon)$, $M_Y(-\alpha(R+\varepsilon))$ are finite,
and hence Assumption \ref{assumption3} is satisfied.

\subsubsection*{Integrability condition of the recurrence times: Assumption \ref{ass:exponentially_rec_time}}
Finally, we have to check if Assumption \ref{ass:exponentially_rec_time} is satisfied.
To do so, by the description of stopping rules $\phi_i$ at the beginning of Subsection \ref{subsec:CL} and \eqref{suff0},
it suffices to show that, for fixed level $\lambda >a$, there exists some $r>R$ such that
\begin{equation}\label{suff}
\mathbb{E}_\lambda^{\mathbb{Q}^{(r)}}\left[e^{\theta(r) S^\lambda_1}\right]<+\infty.
\end{equation}

To prove \eqref{suff}, we will use the ideas of the proof of Theorem \ref{thm:light_tailed_recurrence}, that is, we identify some constant $q$ and $\tilde \lambda < \lambda$ such that
\begin{equation*}\label{warunki}
\mathbb{E}^{\mathbb{Q}^{(r)}}_\lambda \left[e^{q S_1^{\tilde \lambda}}\right] < +\infty\quad \text{and}\quad \mathbb{E}^{\mathbb{Q}^{(r)}}_{\tilde \lambda} \left[e^{q S_1^{ \lambda}}\right] < +\infty.
\end{equation*}
For this, using the proof of Theorem \ref{thm:light_tailed_recurrence}, we recall that, under the stationary regime, the recurrence time $S^\lambda_1$ is light-tailed and $\mathbb{E}_\nu^{\mathbb{Q}^{(r)}}\left[e^{q S^\lambda_1}\right]<+\infty$ for all $q < \beta (\lambda-a) p^{(r)}(\lambda)$, where $ p^{(r)}(\lambda)$ denotes the density of the stationary distribution under the measure $\mathbb{Q}^{(r)}$.

Due to the proof of Theorem \ref{thm:light_tailed_recurrence}, we have that for almost all $\tilde \lambda < \lambda$, that $\mathbb{E}_\lambda^{\mathbb{Q}^{(r)}}\left[e^{q S_1^\lambda}\right]< +\infty$, where $q < \min \left( \beta (\lambda-a) p^{(r)}(\lambda), \beta (\tilde \lambda-a) p^{(r)}(\tilde \lambda)\right)$. Unfortunately, we have to show that this holds for $q=\theta(r)$, a quantity depending on $r$. Furthermore, we know that the exponential moment is finite for almost all $\tilde \lambda$, but we do not know which $\tilde \lambda$ does not satisfy this property.
To bypass these problems, we aim to identify a lower bound $K$ for $\beta (\tilde \lambda-a) p^{(r)}(\tilde \lambda)$ which is independent of $r$ and holds uniformly for $\tilde \lambda \in I$, where $I$ is an interval containing $\lambda$. This would give us $\mathbb{E}_\lambda^{\mathbb{Q}^{(r)}}\left[e^{K S_1^\lambda}\right]< +\infty$, for all $r$, which would allow us to choose $r>R$ such that $\theta(r)< K$. Consequently, the necessary integrability condition \eqref{suff} will be satisfied.

From the proof of Lemma \ref{Qrparamneters}
it follows that,
under a measure $\mathbb{Q}^{(r)}$ for some arbitrary $r< r_{\rm max}$,
the intensity process $\{\lambda_t\}_{t\geq 0}$ of a Markovian Hawkes process $\{N_t\}_{t\geq 0}$ with the baseline intensity $a M_U(r)M_Y(-\alpha(r))$, decay parameter $\beta$ and shocks of the form
$M_U(r)M_Y(-\alpha(r)) \tilde Y$, where $\tilde Y$ has distribution
\[\frac{e^{-\alpha(r) y}}{M_Y(-\alpha(r))} F_Y(\mathrm{d}y) = \frac{1}{\gamma+\alpha(r)} e^{-(\gamma + \alpha(r))y} \, \mathrm{d}y,\]
that is, with
the shocks that are exponentially distributed with the parameter
\[\gamma^{(r)}:= \frac{\gamma + \alpha(r)}{M_U(r) M_Y(-\alpha(r))} = \frac{(\gamma+\alpha(r))^2 (\mu-r)}{\gamma \mu}. \]
Hence, we can use the already determined stationary distribution inn \eqref{densityformexp}
for Markovian Hawkes intensities with exponentially distributed shocks to conclude that our process $\{ \lambda_t\}_{t \geq 0}$ has stationary density
\begin{align*}
    &p^{(r)}(\lambda ) = (\lambda-a)^{\frac{M_U(r)M_Y(-\alpha(r)) a}{\beta}-1}\\&\quad \times \exp\left(- \frac{(\beta \gamma^{(r)} -1) M_U(r)M_Y(-\alpha(r))}{\beta} (\lambda-a) \right)
     \\ &\quad \times (M_U(r)M_Y(-\alpha(r)))^{\frac{M_U(r)M_Y(-\alpha(r)) a}{\beta}-1} \left(\frac{\beta \gamma^{(r)} -1}{\beta}\right)^{\frac{a M_U(r)M_Y(-\alpha(r))}{\beta}} \\&\quad\times \Gamma\left( \frac{M_U(r)M_Y(-\alpha(r))a}{\beta}\right)^{-1}.
\end{align*}
Choose now some $\varepsilon>0$ such that $r_{\rm max}>R+\varepsilon$. Then, we see that the function $p^{(r)}(\lambda )$ is well defined for all $r \in [R,R+\varepsilon]$, continuous as a bivariate function $p(r,\lambda) := p^{(r)}(\lambda)$, and strictly positive. Consequently, we find that this function is uniformly bounded from below on $[R,R+\varepsilon] \times [\frac{\lambda+a}{2},\lambda]$ by some positive constant $K$.

Recall that the function $\theta(r)$ is continuous and the adjustment condition $R$ satisfies $\theta(R)=0$. By this, we can choose some $r \in [R,R+\varepsilon]$ such that $\theta(r) < K \beta \left(\frac{\lambda-a}{2}\right)$. Therefore, it holds for this specific $r$ that  \[ \mathbb{E}^{\mathbb{Q}^{(\bar r)}}_{ \lambda} \left[ e^{\theta(\bar r) S^{ \lambda}_1} \right] < +\infty.\] Consequently, we have that all our assumptions are satisfied and, assuming \eqref{additionalassumptionexp}, from Theorem \ref{cramerasymptotics} we can conclude that
\[ \lim_{u\to+\infty}\psi(u, \tilde \lambda) e^{Ru} = C^{\tilde{\lambda}},\]
for a positive constant $C^{\tilde{\lambda}}$.


\section*{Acknowledgments}
This research was funded in whole, or in part, by the Austrian Science Fund (FWF) P 33317. For the purpose of open access, the author has applied a CC BY public copyright licence to any Author Accepted Manuscript version arising from this submission.

Z. Palmowski acknowledges that the research is partially supported by Polish National Science Centre Grant No. 2021/41/B/HS4/00599.

\small
\bibliographystyle{abbrv}
\bibliography{sn}

\end{document}